\documentclass[12pt]{article}

\usepackage[margin=1in]{geometry}
\usepackage{bm}
\usepackage{graphicx}
\usepackage[colorlinks=true,linkcolor=blue,citecolor=blue,urlcolor=blue]{hyperref}
\usepackage{amsmath,amssymb,amsthm,mathtools,bm}
\usepackage{graphicx}
\usepackage{subcaption}
\usepackage{booktabs}
\usepackage{algorithm}
\usepackage{algpseudocode}
\usepackage{float}
\usepackage{stmaryrd}
\usepackage[mathscr]{eucal}
\usepackage{tikz}
\usetikzlibrary{arrows.meta,positioning,calc}
\usepackage{multirow}
\usepackage{indentfirst}

\newtheorem{theorem}{Theorem}[section]

\newtheorem{proposition}[theorem]{Proposition}

\newtheorem{definition}[theorem]{Definition}

\newcommand{\R}{\mathbb{R}}
\newcommand{\N}{\mathbb{N}}
\newcommand{\PP}{\mathcal{P}}

\newcommand{\dd}{\,\mathrm{d}}

\title{A Measure-Consistent Operator Learning Method for Infinite-Dimensional Master Equations}
\author{%
  Chenyao Wang$^{1,2}$, Hongyu Liu$^{2}$, Hui Liang$^{1,*}$ \\[0.5em]
  \parbox{0.92\textwidth}{\centering
    $^{1}$School of Science, Harbin Institute of Technology, Shenzhen, 518055 China\\
    $^{2}$Department of Mathematics, City University of Hong Kong, Hong Kong Kowloon, China\\
    Correspondence: lianghui@hit.edu.cn
  }
}
\date{June 6, 2026}

\begin{document}

\maketitle

\begin{abstract}
Master equations in mean field game theory characterize feedback value functions that depend on time, state (space), and the population distribution. Their numerical approximation is challenging because the unknown is defined on a space of probability measures and the equation involves intrinsic measure derivatives and nonlocal population terms. This paper proposes a measure-consistent operator learning method (MCOL) for infinite-dimensional master equations. The population distribution is represented by an empirical measure and encoded through a symmetric pooling structure, so that the network input is built directly from the particles representing the measure. The same particles are used in the empirical quadrature of the nonlocal residual terms, avoiding additional quadrature grids or auxiliary integration points. A key feature is that the intrinsic derivative appearing in the residual is induced by the same measure-dependent representation that defines the approximation of the value function. Consequently, the value function, its measure derivative, and the empirical residual are tied to a common measure representation, leading to a structurally coupled value-derivative approximation. We also introduce an error decomposition separating neural approximation error from empirical discretization error. Numerical experiments on several master equations show that MCOL accurately approximates the value function, intrinsic measure derivatives, and feedback quantities, and remains robust under changes in the input measures.
\end{abstract}


\section{Introduction}
\label{sec:introduction}

Mean field games (MFGs), introduced independently by Huang, Caines and Malham\'e and by Lasry and Lions, provide a mathematical framework for the analysis of strategic interactions among a large number of weakly coupled agents \cite{Huang2006,LasryLions2007}. In the mean-field limit, an individual agent responds to the aggregate population distribution rather than to the states of all other agents separately. This formulation leads to tractable limiting models for large-population Nash equilibria and has been used in crowd dynamics, economics, finance and engineering systems \cite{Achdou20201,CarmonaFouqueSun2015,GuLauriereMerkelPayne2024}.

A central object in MFGs is the master equation. While a classical MFG system describes the equilibrium associated with a prescribed initial distribution, the master equation characterizes the feedback value function of a representative agent as a function of time, state, and the current population distribution. The unknown is an infinite-dimensional feedback value function of the form 
\[ U:[0,T]\times \Omega \times \mathcal{P}(\Omega)\to \mathbb{R}, \]
where \(\Omega\subseteq \mathbb{R}^d\) and \(\mathcal{P}(\Omega)\) denotes the set of Borel probability measures on \(\Omega\). The master equation can therefore be regarded as a Hamilton--Jacobi equation posed on a space of probability measures, and it encodes a family of MFG systems in a single feedback object \cite{Cardaliaguet2019}.

The analysis of master equations has developed substantially in recent years. Existing results address well-posedness, regularity, differentiability with respect to probability measures, stability, and the rigorous connection between finite-player Nash systems, classical MFG systems, and the limiting master equation \cite{Cardaliaguet2019,CarmonaDelarue2018I,GangboMeszarosMouZhang2022}. These results provide the theoretical foundation for the master equation formulation. From the computational viewpoint, however, the direct approximation of continuous-state master equations remains substantially more difficult than the numerical solution of classical MFG systems.

Most conventional numerical methods for MFGs are designed for the coupled Hamilton--Jacobi--Bellman and Fokker--Planck system associated with a fixed initial distribution. Representative approaches include finite difference methods and monotone schemes, semi-Lagrangian methods, variational formulations, particle methods, and stochastic numerical methods \cite{Achdou2010,achdou2012mean,andreev2017preconditioning,liu2021computational,yu2023computational}. These methods have been successful for many MFG systems, but they do not directly resolve the full feedback dependence of \(U(t,x,m)\) on arbitrary population distributions.

Machine learning methods have recently provided flexible tools for high-dimensional and parametric PDEs. Physics-informed neural networks approximate PDE solutions by minimizing differential residuals together with boundary or terminal losses \cite{Raissi2019}. Neural operators, including DeepONet and Fourier neural operators, learn nonlinear mappings between function spaces \cite{Lu2021,Li2020}, while physics-informed DeepONets further incorporate PDE constraints into operator learning for parametric PDEs \cite{WangWangPerdikaris2021}. More recently, adaptive coordinate transforms have been introduced into neural operators to reduce spatial misalignment in evolving PDE fields \cite{liu2026adaptive}. Recent work has also explored more structured designs neural network for scientific computing. Cai et al. studied permutation-trained networks with universal approximation guarantees \cite{cai2025neural}, and Chen et al. proposed the Sidecar framework for structure-preserving neural PDE solvers \cite{chen2025structure}. For problems involving probability measures, Pham and Warin introduced mean-field neural networks for learning mappings on Wasserstein spaces \cite{PhamWarin2023}. 

In the MFG and mean field control literature, neural methods have been developed to overcome the limitations of grid-based discretizations in high dimensions. Ruthotto et al. proposed a framework based on Lagrangian and Eulerian formulations \cite{Ruthotto2020}, while Lin et al. exploited the primal--dual structure of stochastic MFGs through alternating population and control networks \cite{lin2021alternating}. Fang et al. proposed a regenerative deep policy iteration method for high-dimensional finite horizon MFGs, avoiding the direct solution of the coupled HJB--FP system through particle-based measure updates and policy iteration \cite{fang2026deep}. Xu et al. developed an online interactive physics-informed diffusion-adversarial network for MFG systems \cite{Xu2026}. In parallel, Huang and Lai studied unsupervised operator learning for MFGs, aiming to map problem instances directly to their corresponding solutions \cite{Huang2025}.

These developments do not remove the main structural difficulty of master equations. The measure variable is not a finite-dimensional parameter, and the equation contains intrinsic derivatives with respect to this measure variable. In particular, a residual evaluation requires a consistent approximation of the value function \(U(t,x,m)\), the intrinsic derivative \(D_mU(t,x,m,y)\), and nonlocal terms such as
\[
\int_\Omega \operatorname{div}_y[D_mU(t,x,m,y)]\,\dd m(y),
\qquad
\int_\Omega D_mU(t,x,m,y)\cdot D_pH(y,D_xU(t,y,m))\,\dd m(y).
\]
Machine learning methods for finite-state master equations avoid part of this difficulty because the population distribution belongs to a finite-dimensional probability simplex \cite{CohenLauriereZell2024}. The relation between finite-state and continuous-state master equations has been studied through convergence analysis as the number of states tends to infinity \cite{BertucciCecchin2024}. There are also application-oriented numerical approaches for particular master equations, such as the semi-Lagrangian neural-network method for the Krusell--Smith model \cite{AchdouLasryLions2022}. Nevertheless, a general structure-preserving numerical framework for continuous-state master equations remains comparatively underdeveloped.

This study addresses this issue by proposing a measure-consistent operator learning method, abbreviated as MCOL, for infinite-dimensional master equations. The method is built around two principles. First, the population input is represented by an empirical measure and encoded by a symmetric pooling map. Second, the intrinsic derivative \(D_mU\) used in the master-equation residual is induced from the same neural representation as \(U\). The same empirical measure is
then used in the quadrature of the nonlocal residual terms. Thus the
approximation of \(U_\Theta\), the induced measure derivative \(D_mU_\Theta\),
and the empirical residual are tied to one common measure representation.

The main contributions are as follows.

1. A measure-consistent empirical operator learning framework for master equations: The probability measure is represented by empirical particles and encoded through a permutation-invariant pooling map.

2. An induced construction of the intrinsic derivative for the value function: The measure derivative is obtained from the value function approximation itself, yielding a structurally coupled representation of the value function and its derivative for the construction of the residual.

3. A physics-informed empirical residual formulation: The nonlocal terms
involving \(D_mU\), \(\operatorname{div}_yD_mU\), and related derivative
quantities are evaluated by empirical quadrature and automatic differentiation.

The remainder of the paper is organized as follows. Section~\ref{sec:problem_setting} formulates the master equation and recalls the notation for measure derivatives used throughout the paper. Section~\ref{sec:methodology} develops the proposed MCOL framework, including the measure-consistent representation, the induced intrinsic derivative, the empirical residual, the grouped training strategy, and the error decomposition used for evaluation. Section~\ref{sec:numerical_experiments} presents numerical experiments on 1D and 2D state-space problems, a characteristic relation test along an independently computed MFG trajectory, and a systemic-risk problem with common noise. Section~\ref{sec:conclusion} concludes the paper. The comparison baseline is described in Appendix~\ref{sec:baseline}.

\section{Problem statement}
\label{sec:problem_setting}

\subsection{Measure derivatives on $\PP(\Omega)$}
\label{subsec:measure_calculus}

We recall the measure derivative notation used throughout the paper, following the standard formulation in \cite{Cardaliaguet2019,ricciardi2022master,liu2025inverse}. Let $\Omega\subset\mathbb{R}^d$ be a bounded domain with sufficiently smooth boundary, and denote by $\mathcal P(\Omega)$ the space of Borel probability measures on $\Omega$.

\begin{definition}\label{def:first_variational_derivative}
	A function $U:\PP(\Omega)\to\R$ is said to be of class $C^1$ if there exists a continuous map
	\[
	K:\PP(\Omega)\times \Omega\to\R
	\]
	such that, for every $m_1,m_2\in\PP(\Omega)$,
	\begin{equation}\label{eq:Omega-C1-def-polished}
		\lim_{s\to 0^+}\frac{U\bigl(m_1+s(m_2-m_1)\bigr)-U(m_1)}{s}
		=
		\int_{\Omega} K(m_1,y)\,\dd(m_2-m_1)(y).
	\end{equation}
	The map $K$ is unique only up to an additive constant. We denote by $\frac{\delta U}{\delta m}$ the normalized representative satisfying \eqref{eq:Omega-C1-def-polished} together with
	\begin{equation*}
		\int_{\Omega} \frac{\delta U}{\delta m}(m,y)\,\dd m(y)=0.
	\end{equation*}
\end{definition}

\begin{definition}\label{def:intrinsic_derivative}
	Assume that $U$ is of class $C^1$ and that $\frac{\delta U}{\delta m}(m,\cdot)$ is $C^1$ in the second variable. The intrinsic derivative $D_mU$ is defined by
	\begin{equation*}
		D_mU(m,y):=D_y\!\left(\frac{\delta U}{\delta m}(m,y)\right),
		\qquad (m,y)\in \PP(\Omega)\times \Omega.
	\end{equation*}
\end{definition}

\subsection{master equation}
\label{subsec:master_equation}

We first present a master equation with homogeneous Neumann and no-flux boundary conditions, which fixes the notation and residual form used in the method:
\begin{equation}
	\label{eq:master_general_polished}
	\left\{
	\begin{aligned}
		&-\partial_t U(t,x,m)-\Delta_x U(t,x,m)+H\bigl(x,D_xU(t,x,m)\bigr)
		-\int_{\Omega} \operatorname{div}_y\!\bigl[D_mU(t,x,m,y)\bigr] \,\dd m(y) \\
		&\qquad\qquad
		+\int_{\Omega} D_mU(t,x,m,y)\cdot D_pH\bigl(y,D_xU(t,y,m)\bigr)\,\dd m(y)
		=F(x,m), \\
		&\hspace{8.2cm}(t,x,m)\in (0,T)\times \Omega\times \PP(\Omega), \\
		&D_xU(t,x,m)\cdot n(x)=0,
		\qquad (t,x,m)\in (0,T)\times \partial\Omega\times \PP(\Omega), \\
		&D_mU(t,x,m,y)\cdot n(y)=0,
		\qquad (t,x,m,y)\in (0,T)\times \Omega\times \PP(\Omega)\times \partial\Omega, \\
		&U(T,x,m)=G(x,m),
		\qquad (x,m)\in \Omega\times \PP(\Omega).
	\end{aligned}
	\right.
\end{equation}
Here \(n\) denotes the outward unit normal on \(\partial\Omega\). The value
function \(U(t,x,m)\) represents the optimal value for a representative agent
at time \(t\), with state \(x\), when the population distribution is \(m\).
The function \(F\) is the running cost, \(G\) is the terminal cost, and \(H\)
is the Hamiltonian. The notation \(D_pH\) denotes the derivative of \(H\) with
respect to its momentum variable. The Neumann condition in the state variable \(x\) and the no-flux condition in
the \(y\) variable in \(D_mU\) are natural when the state process is confined
to \(\Omega\). Well-posedness results for related bounded-domain
master equations under suitable regularity and monotonicity assumptions can be
found in \cite{ricciardi2022master}. Other boundary conditions can also be used, depending on the model and the numerical test. 

Although the master equation is formulated on \(\PP(\Omega)\), many mean-field
applications describe population distributions by densities. In the numerical
experiments, we therefore use probability measures with smooth positive
densities. Throughout the paper, \(m\in\PP(\Omega)\) denotes a probability
measure. If \(m\) is absolutely continuous with respect to the Lebesgue measure
on \(\Omega\), we write
\[
\dd m(y)=\rho(y)\,\dd y,
\qquad
\rho\ge 0,
\qquad
\int_\Omega \rho(y)\,\dd y=1.
\]
Thus, \(\rho\) denotes the density of \(m\) in the absolutely continuous case,
rather than an additional measure variable.

\section{Methodology}
\label{sec:methodology}

This section presents the MCOL for
approximating the solution map \((t,x,m)\mapsto U(t,x,m)\) of
\eqref{eq:master_general_polished}. The construction starts from the residual
requirements of the master equation: one must approximate \(U\), differentiate
it with respect to \((t,x)\), construct the intrinsic derivative
\(D_mU(t,x,m,y)\), and evaluate the nonlocal measure terms. The key point is
that the measure derivative is generated by differentiating the same measure-dependent representation that defines the value function. The resulting computational chain is
\begin{equation*}
	m
	\longmapsto
	m^N
	\longmapsto
	z_\eta(m^N)
	\longmapsto
	U_\Theta(t,x,m^N)
	\longmapsto
	D_mU_\Theta(t,x,m^N,y)
	\longmapsto
	\mathcal R_\Theta(t,x,m^N),
\end{equation*}
where \(m^N\) is the empirical approximation of the input measure,
\(z_\eta(m^N)\) is the pooled measure embedding, \(y\in\Omega\) is the variable of the intrinsic derivative, and
\(\mathcal R_\Theta\) denotes the sampled residual.

\subsection{Measure-consistent representation and the induced intrinsic derivative}
\label{subsec:single_network_dm}

For \(N\in\N\) particles \(\xi_i\in\Omega\), the empirical input is
\begin{equation}
	\label{eq:empirical_measure_polished}
	m^N=\frac1N\sum_{i=1}^N\delta_{\xi_i}.
\end{equation}
Integrals with respect to \(m^N\) are evaluated by particle averages. The
empirical measure is encoded by the symmetric feature average
\begin{equation}
	\label{eq:measure_encoding_polished}
	z_\eta(m^N)
	:=
	\frac1N\sum_{i=1}^N\phi_\eta(\xi_i)\in\R^h,
\end{equation}
where \(\phi_\eta:\Omega\to\R^h\) is the particle feature map and
\(h\) is the feature dimension. Motivated by the branch--trunk architecture of DeepONet \cite{Lu2021} and the structure-preserving framework \cite{chen2025structure}, we use a branch--trunk pairing to represent the dependence of \(U\) on the space-time
variables and on the measure argument. In contrast to a standard DeepONet, where the branch input is typically a
function represented by its values at fixed sensor points, the branch input here
is an empirical probability measure \(m^N\). It is therefore encoded through the
symmetric average \eqref{eq:measure_encoding_polished}, which provides a
measure-dependent differentiable structure from which the intrinsic derivative
\(D_mU_\Theta\) can be derived. Let \(\psi_\eta:\R^h\to\R^q\) be the branch
aggregation map, let \(T_\theta:[0,T]\times\Omega\to\R^q\) be the trunk map, and
let \(q\) be the branch--trunk feature dimension. Here \(\eta\) collects the
branch parameters, \(\theta\) denotes the trunk parameters, and \(b\in\R\) is a
scalar bias. The MCOL approximation is
\begin{equation}
	\label{eq:branch_trunk_ansatz}
	U_\Theta(t,x,m^N)
	=
	\left\langle
	T_\theta(t,x),\psi_\eta\bigl(z_\eta(m^N)\bigr)
	\right\rangle+b,
	\qquad
	\Theta=(\theta,\eta,b).
\end{equation}
This architecture is illustrated in Figure~\ref{fig:stream}(a). The measure variable is the input, whereas
\((t,x)\) is the evaluation coordinate.

\begin{figure}[htbp]
	\centering
	\includegraphics[width=0.95\textwidth]{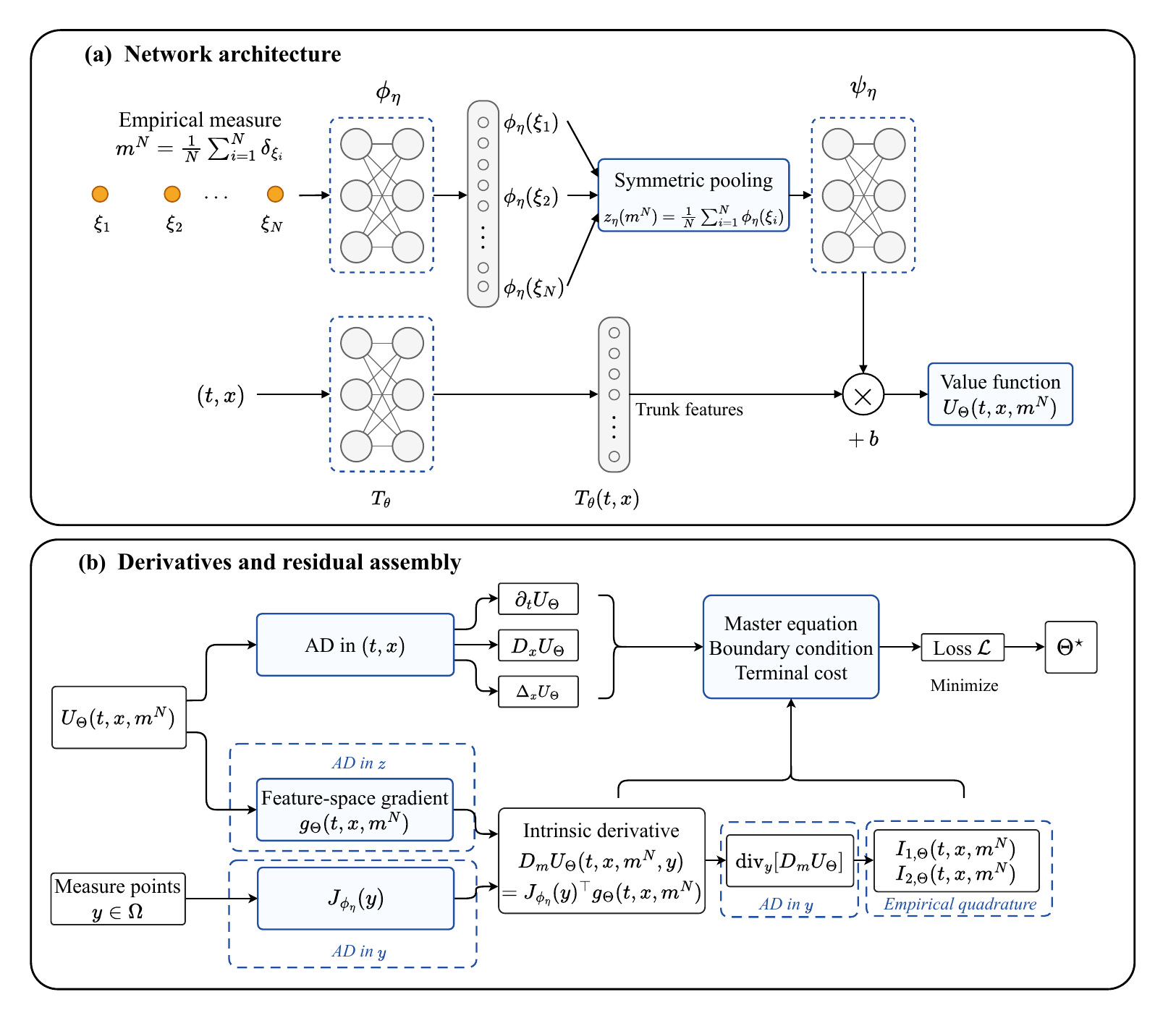}
	\caption{
		MCOL architecture with induced intrinsic derivatives and residual assembly.
	}
	\label{fig:stream}
\end{figure}

To derive the measure derivative, introduce the continuous-measure
extension of \eqref{eq:measure_encoding_polished},
\begin{equation*}
	z_\eta(m):=\int_\Omega \phi_\eta(y)\,\dd m(y),
\end{equation*}
and write the architecture as
\begin{equation*}
	U_\Theta(t,x,m)
	=
	\mathcal U_\Theta\bigl(t,x,z_\eta(m)\bigr),
	\qquad
	\mathcal U_\Theta(t,x,r)
	=
	\left\langle T_\theta(t,x),\psi_\eta(r)\right\rangle+b,
	\quad r\in\R^h.
\end{equation*}
The feature-space gradient is
\begin{equation*}
	g_\Theta(t,x,m)
	:=
	D_r\mathcal U_\Theta(t,x,r)\big|_{r=z_\eta(m)}
	\in\R^h,
\end{equation*}
where \(D_r\) denotes differentiation with respect to the feature variable
\(r\). For the branch--trunk form \eqref{eq:branch_trunk_ansatz},
\begin{equation*}
	g_\Theta(t,x,m)
	=
	J_{\psi_\eta}\bigl(z_\eta(m)\bigr)^\top T_\theta(t,x),
\end{equation*}
where \(J_{\psi_\eta}\) is the Jacobian of \(\psi_\eta\). The following result is
the structural point of the method: the intrinsic derivative used in the
residual is fully determined by the same architecture that approximates \(U\).

\begin{proposition}
	\label{prop:single_network_construction_rigorous}
	Assume that \(\phi_\eta\in C^1(\overline\Omega;\mathbb R^h)\) and that
	\(\mathcal U_\Theta\) is continuously differentiable with respect to its feature
	variable. For
	\[
	U_\Theta(t,x,m)
	=
	\mathcal U_\Theta\bigl(t,x,z_\eta(m)\bigr),
	\qquad
	z_\eta(m)=\int_\Omega \phi_\eta(y)\,\mathrm dm(y),
	\]
	define
	\[
	g_\Theta(t,x,m)
	:=
	D_r\mathcal U_\Theta(t,x,r)\big|_{r=z_\eta(m)} .
	\]
	Then \(U_\Theta(t,x,\cdot)\) admits a first variational derivative, and its
	intrinsic derivative is
	\begin{equation}
		\label{eq:intrinsic_derivative_rigorous}
		D_mU_\Theta(t,x,m,y)
		=
		J_{\phi_\eta}(y)^\top g_\Theta(t,x,m).
	\end{equation}
	In particular, for the branch--trunk representation
	\[
	U_\Theta(t,x,m)
	=
	\left\langle T_\theta(t,x),
	\psi_\eta\bigl(z_\eta(m)\bigr)
	\right\rangle+b,
	\]
	one has
	\begin{equation}
		\label{eq:intrinsic_derivative_explicit_rigorous}
		D_mU_\Theta(t,x,m,y)
		=
		J_{\phi_\eta}(y)^\top
		J_{\psi_\eta}\bigl(z_\eta(m)\bigr)^\top
		T_\theta(t,x).
	\end{equation}
\end{proposition}

\begin{proof}
	Fix \((t,x)\) and let
	\[
	V_{t,x}(m):=U_\Theta(t,x,m).
	\]
	For \(m,m'\in\mathcal P(\Omega)\), set
	\[
	m_s=(1-s)m+sm',
	\qquad s\in[0,1].
	\]
	Since \(z_\eta\) is linear in \(m\),
	\[
	z_\eta(m_s)
	=
	z_\eta(m)
	+
	s\int_\Omega \phi_\eta(y)\,\mathrm d(m'-m)(y).
	\]
	Therefore, by the chain rule,
	\[
	\begin{aligned}
		\lim_{s\to0^+}
		\frac{V_{t,x}(m_s)-V_{t,x}(m)}{s}
		&=
		g_\Theta(t,x,m)
		\cdot
		\int_\Omega \phi_\eta(y)\,\mathrm d(m'-m)(y) \\
		&=
		\int_\Omega
		\phi_\eta(y)\cdot g_\Theta(t,x,m)\,
		\mathrm d(m'-m)(y).
	\end{aligned}
	\]
	Thus a representative of the first variational derivative is
	\[
	\phi_\eta(y)\cdot g_\Theta(t,x,m).
	\]
	A normalized representative is obtained by subtracting its \(m\)-average:
	\[
	\frac{\delta U_\Theta}{\delta m}(t,x,m,y)
	=
	\phi_\eta(y)\cdot g_\Theta(t,x,m)
	-
	\int_\Omega
	\phi_\eta(z)\cdot g_\Theta(t,x,m)\,\mathrm dm(z).
	\]
	Differentiating this expression with respect to \(y\), the normalization term
	drops out, and hence
	\[
	D_mU_\Theta(t,x,m,y)
	=
	J_{\phi_\eta}(y)^\top g_\Theta(t,x,m).
	\]
	For the branch--trunk form,
	\[
	g_\Theta(t,x,m)
	=
	J_{\psi_\eta}\bigl(z_\eta(m)\bigr)^\top T_\theta(t,x),
	\]
	which gives \eqref{eq:intrinsic_derivative_explicit_rigorous}.
\end{proof}

Proposition~\ref{prop:single_network_construction_rigorous} is the structural
basis of measure consistency. Once \(U_\Theta\) is fixed, the corresponding
\(D_mU_\Theta\) is fixed by \eqref{eq:intrinsic_derivative_rigorous}. All measure derivatives used below are obtained by the same architecture of the value function. In practice, \(g_\Theta\) is obtained by automatic differentiation with
respect to the pooled feature variable, and \(D_mU_\Theta\) is then obtained by
differentiating the particle feature map with respect to the independent variable \(y\).

The same construction also satisfies the empirical chain rule associated with
particle lifts of functions on probability measures. For the residual terms involving \(\operatorname{div}_yD_mU_\Theta\), we assume
that \(\phi_\eta\) is twice continuously differentiable. In the common-noise example of Subsection~\ref{subsec:systemic_risk_common_noise}, derivatives with respect to both the state and the measure, as well as second variations, are required. Accordingly, the trunk map, the aggregation map, and the particle feature map are assumed to be sufficiently smooth, which is ensured in the implementation by using smooth activation functions such as \(\tanh\). The construction defines \(D_mU_\Theta\) intrinsically from the network
representation. We next record a related particle chain rule, which clarifies
how this derivative is reflected in the lifted empirical measure representation,
although it is not used to compute \(D_mU_\Theta\) in the MCOL residual.

\begin{proposition}
	\label{prop:mcol_particle_chain_rule}
	Let \(m^N=N^{-1}\sum_{i=1}^N\delta_{\xi_i}\) and define the lifted function
	\[
	\widetilde U_\Theta(t,x,\xi_1,\ldots,\xi_N)
	:=
	U_\Theta(t,x,m^N).
	\]
	Under the assumptions of Proposition~\ref{prop:single_network_construction_rigorous},
	for each \(i=1,\ldots,N\),
	\begin{equation}
		\label{eq:mcol_particle_chain_rule}
		\nabla_{\xi_i}\widetilde U_\Theta(t,x,\xi_1,\ldots,\xi_N)
		=
		\frac1N D_mU_\Theta(t,x,m^N,\xi_i).
	\end{equation}
\end{proposition}

\begin{proof}
	Fix \((t,x)\) and view \(\widetilde U_\Theta\) as the particle lift of the
	measure functional \(V_{t,x}(m)=U_\Theta(t,x,m)\). By the empirical encoding
	\eqref{eq:measure_encoding_polished},
	\[
	z_\eta(m^N)=\frac1N\sum_{j=1}^N\phi_\eta(\xi_j).
	\]
	Therefore the ordinary chain rule in the particle coordinate \(\xi_i\) gives
	\[
	\nabla_{\xi_i}\widetilde U_\Theta
	=
	\frac1N J_{\phi_\eta}(\xi_i)^\top
	D_r\mathcal U_\Theta(t,x,r)\big|_{r=z_\eta(m^N)}.
	\]
	On the other hand, Proposition~\ref{prop:single_network_construction_rigorous},
	which was obtained from Definitions~\ref{def:first_variational_derivative}
	and~\ref{def:intrinsic_derivative}, gives
	\[
	D_mU_\Theta(t,x,m^N,\xi_i)
	=
	J_{\phi_\eta}(\xi_i)^\top
	D_r\mathcal U_\Theta(t,x,r)\big|_{r=z_\eta(m^N)}.
	\]
	Combining the last two identities yields
	\eqref{eq:mcol_particle_chain_rule}.
\end{proof}

In MCOL, the particles serve only as a finite
representation of the empirical measure, and the particle gradient relation
\eqref{eq:mcol_particle_chain_rule} is induced by the architecture itself.
By contrast, the baseline method that approximates \(U\) and \(D_mU\) by two
separate networks can enforce this relation only weakly, for instance through
an additional loss. This comparison is discussed in
\ref{sec:baseline}.

We stress that Proposition~\ref{prop:mcol_particle_chain_rule} is not used as
the computational definition of \(D_mU_\Theta\) in the PDE residual. Indeed,
\eqref{eq:mcol_particle_chain_rule} characterizes
\(D_mU_\Theta(t,x,m^N,\xi_i)\) only at the support points of the empirical
measure, through the gradient of the lifted finite-particle function. The PDE
residual and the boundary losses, however, require
\(D_mU_\Theta(t,x,m^N,y)\) and
\(\operatorname{div}_y[D_mU_\Theta](t,x,m^N,y)\) at general points
\(y\in\Omega\), including points on \(\partial\Omega\). Therefore, in the
implementation, \(D_mU_\Theta\) is evaluated from the intrinsic derivative
induced by the network, as given in
\eqref{eq:intrinsic_derivative_rigorous}--\eqref{eq:dm_ad_residual}.
Proposition~\ref{prop:mcol_particle_chain_rule} shows that the intrinsic derivative constructed above agrees with the particle lift when evaluated at the empirical support.

\subsection{Empirical representation of probability measures}
\label{subsec:continuous_to_empirical_measures}

The approximation \eqref{eq:branch_trunk_ansatz} is evaluated on empirical measures of
the form \eqref{eq:empirical_measure_polished}. When the available input is an
absolutely continuous probability measure
\[
\dd m(y)=\rho(y)\,\dd y,
\qquad
y\in\Omega\subset\mathbb R^d,
\]
we first convert the density \(\rho\) into particles
\(\{\xi_i\}_{i=1}^N\subset\Omega\), and then use
\[
m^N=\frac1N\sum_{i=1}^N\delta_{\xi_i}
\]
as the network input. The
network does not take the density values of \(\rho\) as input; it only uses the
particle cloud through the symmetric embedding \(z_\eta(m^N)\). Hence, any
density from which particles can be sampled or deterministically constructed can
be used to form an admissible MCOL input.

In the numerical experiments, smooth positive training densities are generated
from Gaussian random fields (GRF). Specifically, after sampling a mean-zero Gaussian
random field \(g\) with a squared-exponential covariance kernel, we define
\[
\widetilde\rho(y)=\exp(\alpha g(y)),
\qquad
\rho(y)=
\frac{\widetilde\rho(y)}
{\int_\Omega \widetilde\rho(s)\,\dd s},
\]
where \(\alpha\ge0\) controls the amplitude of the log-density fluctuation. This
GRF construction is used only to provide diverse training measures and is not a
restriction of the method.

For \(d\)-dimensional domains, we use low-discrepancy reference points and a
Rosenblatt-type inverse transform. For simplicity, consider
\(\Omega=[0,1]^d\). Let
\[
u_j=(u_{j,1},\ldots,u_{j,d})\in[0,1]^d,
\qquad j=1,\ldots,N,
\]
be Sobol points. These points are mapped to particles
\[
\xi_j=(\xi_{j,1},\ldots,\xi_{j,d})
\]
according to the target density \(\rho\). Define the marginal densities
\[
\rho_{1:k}(y_1,\ldots,y_k)
=
\int_{[0,1]^{d-k}}
\rho(y_1,\ldots,y_k,s_{k+1},\ldots,s_d)
\,\dd s_{k+1}\cdots\dd s_d,
\qquad k=1,\ldots,d,
\]
with \(\rho_{1:d}=\rho\). The first coordinate is obtained from the marginal
cumulative distribution function (CDF)
\[
F_1(r)=\int_0^r\rho_1(s)\,\dd s,
\qquad
\xi_{j,1}=F_1^{-1}(u_{j,1}).
\]
For \(k=2,\ldots,d\), after
\(\xi_{j,1},\ldots,\xi_{j,k-1}\) have been determined, we define the conditional
density
\[
\rho_{k|1:k-1}
(y_k\mid \xi_{j,1},\ldots,\xi_{j,k-1})
=
\frac{
	\rho_{1:k}(\xi_{j,1},\ldots,\xi_{j,k-1},y_k)
}{
	\rho_{1:k-1}(\xi_{j,1},\ldots,\xi_{j,k-1})
}.
\]
The corresponding conditional CDF is
\[
F_{k|1:k-1}
(r\mid \xi_{j,1},\ldots,\xi_{j,k-1})
=
\int_0^r
\rho_{k|1:k-1}
(s\mid \xi_{j,1},\ldots,\xi_{j,k-1})\,\dd s,
\]
and the \(k\)-th coordinate is set by
\[
\xi_{j,k}
=
F_{k|1:k-1}^{-1}
(u_{j,k}\mid \xi_{j,1},\ldots,\xi_{j,k-1}).
\]
In implementation, the marginal and conditional densities are evaluated on the
sampling grid by quadrature and interpolation, followed by normalization. When
\(d=1\), this construction reduces to the usual inverse cumulative distribution function (ICDF) discretization
\(\xi_j=F_\rho^{-1}((j-\frac12)/N)\).

Figure~\ref{fig:density_to_empirical_measure} illustrates the conversion from a
continuous density to its empirical measure. Compared with Monte Carlo
sampling, the construction provides a more regular coverage of
the reference probability space for a fixed particle number \(N\). After the
inverse transform, the particles still represent the target density, while the
empirical averages used in the nonlocal residual terms have reduced sampling
fluctuations. This point is important for MCOL, because the same empirical
measure \(m^N\) is used both in the network input and in the empirical quadrature
of the master-equation residual.

The same construction can also be applied to densities that are not generated from Gaussian random fields. This is tested in
Subsection~\ref{subsec:characteristic_relation}, where the trained model is
evaluated along a non-GRF density trajectory obtained from an independently
solved MFG system.

\begin{figure}[htbp]
	\centering
	\includegraphics[width=0.98\textwidth]{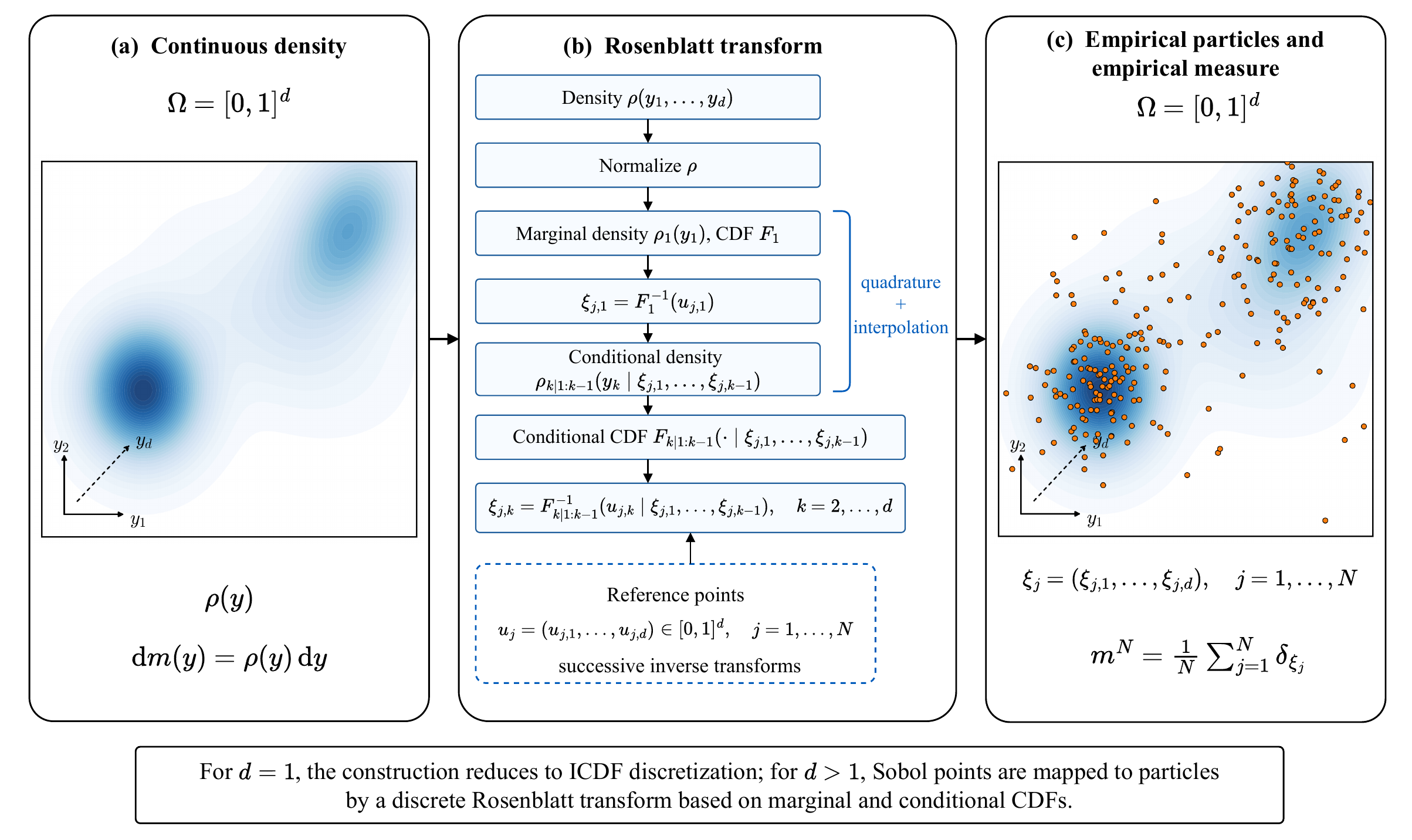}
	\caption{
		Construction of empirical measures from continuous densities.
	}
	\label{fig:density_to_empirical_measure}
\end{figure}

\subsection{Empirical residual and grouped training}
\label{subsec:residual_construction}

We now discretize the residual of \eqref{eq:master_general_polished}. Following the construction of empirical measure in Subsection~\ref{subsec:continuous_to_empirical_measures}, we first generate a set of 10000 measures. In each training mini-batch, \(M_{\mathrm{batch}}\) denotes the number of empirical measures used in one iteration. For these empirical measures, we write
\begin{equation*}
	m^{N,(k)}=\frac1N\sum_{i=1}^N\delta_{\xi_i^{(k)}},
	\qquad
	z_\eta(m^{N,(k)})
	=
	\frac1N\sum_{i=1}^N\phi_\eta(\xi_i^{(k)}),
	\qquad k=1,\ldots,M_{\mathrm{batch}} .
\end{equation*}
Thus, \(M_{\mathrm{batch}}\) controls how many different empirical measures are
sampled in each optimization step. A larger value improves the coverage of the
measure argument in the mini-batch loss, whereas a smaller value reduces the
memory and computational cost of residual evaluation. In the present problem,
the residual of master equation contains several state and measure derivatives,
empirical quadrature terms, and nested automatic differentiation (AD) \cite{baydin2018automatic}. We use a moderate
\(M_{\mathrm{batch}}\) so that each update still contains several independent
empirical measures, while keeping the graph size and GPU memory usage
manageable. This choice provides a practical balance between training
efficiency and accuracy.

All collocation points associated with the same empirical measure share the same
pooled embedding. The residual evaluation uses AD at three levels. Derivatives with respect to \((t,x)\)
give \(\partial_tU_\Theta\), \(D_xU_\Theta\), \(\Delta_xU_\Theta\), and, for
bounded-domain problems, the normal derivative
\(D_xU_\Theta\cdot n\) on \(\partial\Omega\). Derivatives with respect to the
pooled feature variable and the measure variable \(y\) give
\(D_mU_\Theta\) and \(\operatorname{div}_y[D_mU_\Theta]\). After the residual
and loss are assembled, AD is applied to the
full computational graph to compute \(\nabla_\Theta\mathcal L(\Theta)\).

For an evaluation point \((t,x)\), the measure gradient is computed as
\begin{equation*}
	g_\Theta(t,x,m^{N,(k)})
	=
	D_r\mathcal U_\Theta(t,x,r)\big|_{r=z_\eta(m^{N,(k)})}
	=
	\nabla_z\mathcal U_\Theta(t,x,z)\big|_{z=z_\eta(m^{N,(k)})},
\end{equation*}
where \(\nabla_z\) denotes the gradient with respect to the pooled feature
variable. Then, for any \(y\in\Omega\),
\begin{equation}
	\label{eq:dm_ad_residual}
	D_mU_\Theta(t,x,m^{N,(k)},y)
	=
	D_y\!\left[\phi_\eta(y)\cdot g_\Theta(t,x,m^{N,(k)})\right]
	=
	J_{\phi_\eta}(y)^\top g_\Theta(t,x,m^{N,(k)}).
\end{equation}
When evaluating \(D_mU_\Theta(t,x,m^N,\xi_i)\) in the empirical quadrature, the
variable \(y\) is treated as an independent copy of the particle location,
whereas the pooled embedding \(z_\eta(m^N)\) is held fixed with respect to the
\(y\)-differentiation. This distinction is important for computing the
intrinsic derivative rather than the full particle gradient of the lifted
function. The divergence term is obtained by one further differentiation:
\begin{equation*}
	\operatorname{div}_y[D_mU_\Theta](t,x,m^{N,(k)},y)
	=
	\sum_{\ell=1}^d
	\partial_{y_\ell}
	\left(D_mU_\Theta(t,x,m^{N,(k)},y)\right)_\ell .
\end{equation*}
If the particle coordinates are normalized before entering the network, this
normalization is treated as part of \(\phi_\eta\), and the corresponding scaling
factors are included automatically by the chain rule.

The collocation points are grouped by empirical measure. For each selected
empirical measure \(m^{N,(k)}\), we evaluate the PDE residual at scattered
points
\[
(t^{\mathrm{pde}}_{k,q},x^{\mathrm{pde}}_{k,q})
\in (0,T)\times\Omega ,
\]
the terminal condition at points \(x^T_{k,p}\in\Omega\), and, when boundary
conditions are imposed, the corresponding state-boundary residual and
measure-boundary residual. The scattered training points are generated from
Sobol sequences.

Let
\[
(t^{\mathrm{pde}}_{k,q},x^{\mathrm{pde}}_{k,q}),
\qquad q=1,\ldots,Q_{\mathrm{pde}},
\]
be the PDE collocation points associated with \(m^{N,(k)}\), where
\(Q_{\mathrm{pde}}\) is the number of PDE points per empirical measure. The
empirical counterparts of the two nonlocal terms in
\eqref{eq:master_general_polished} are
\begin{equation*}
	I_{1,\Theta}^{(k,q)}
	:=
	\frac1N\sum_{i=1}^N
	\operatorname{div}_y\!\left[
	D_mU_\Theta\bigl(
	t^{\mathrm{pde}}_{k,q},
	x^{\mathrm{pde}}_{k,q},
	m^{N,(k)},y
	\bigr)
	\right]_{y=\xi_i^{(k)}},
\end{equation*}
and
\begin{equation*}
	I_{2,\Theta}^{(k,q)}
	:=
	\frac1N\sum_{i=1}^N
	D_mU_\Theta\bigl(
	t^{\mathrm{pde}}_{k,q},
	x^{\mathrm{pde}}_{k,q},
	m^{N,(k)},
	\xi_i^{(k)}
	\bigr)
	\cdot
	D_pH\!\left(
	\xi_i^{(k)},
	D_xU_\Theta\bigl(
	t^{\mathrm{pde}}_{k,q},
	\xi_i^{(k)},
	m^{N,(k)}
	\bigr)
	\right).
\end{equation*}
Here \(I_{1,\Theta}^{(k,q)}\) approximates
\(\int_\Omega\operatorname{div}_y[D_mU_\Theta]\,\dd m\), while
\(I_{2,\Theta}^{(k,q)}\) approximates the Hamiltonian transport term involving
\(D_mU_\Theta\cdot D_pH\). The empirical quadrature uses the same particles that define the input measure
\(m^{N,(k)}\). Hence, the nonlocal measure terms can be evaluated without
introducing auxiliary integration points, or a
separate numerical integration scheme. This keeps the measure representation in
the network input and in the residual assembly consistent, while reducing the
cost and complexity of the implementation. With these empirical approximations, the PDE residual is
\begin{equation}
	\label{eq:residual_polished}
	\begin{aligned}
		\mathcal R^{\mathrm{pde}}_{\Theta,k,q}
		:={}&
		-\partial_tU_\Theta\bigl(
		t^{\mathrm{pde}}_{k,q},
		x^{\mathrm{pde}}_{k,q},
		m^{N,(k)}
		\bigr)
		-\Delta_xU_\Theta\bigl(
		t^{\mathrm{pde}}_{k,q},
		x^{\mathrm{pde}}_{k,q},
		m^{N,(k)}
		\bigr) \\
		&+
		H\!\left(
		x^{\mathrm{pde}}_{k,q},
		D_xU_\Theta\bigl(
		t^{\mathrm{pde}}_{k,q},
		x^{\mathrm{pde}}_{k,q},
		m^{N,(k)}
		\bigr)
		\right)
		-I_{1,\Theta}^{(k,q)}
		+I_{2,\Theta}^{(k,q)} \\
		&-
		F\bigl(
		x^{\mathrm{pde}}_{k,q},
		m^{N,(k)}
		\bigr).
	\end{aligned}
\end{equation}

The terminal residual uses terminal collocation points
\(x^T_{k,p}\), \(p=1,\ldots,Q_T\), where \(Q_T\) is the number of terminal points
per empirical measure:
\begin{equation*}
	\mathcal R^T_{\Theta,k,p}
	:=
	U_\Theta\bigl(T,x^T_{k,p},m^{N,(k)}\bigr)
	-
	G\bigl(x^T_{k,p},m^{N,(k)}\bigr).
\end{equation*}
For bounded-domain problems with Neumann conditions, we also use
state-boundary points
\((t^{\partial x}_{k,r},x^{\partial x}_{k,r})\),
\(r=1,\ldots,Q_{\partial x}\), and boundary points for the \(y\)-variable
\[
(t^{\partial m}_{k,s},x^{\partial m}_{k,s},y^{\partial m}_{k,s}),
\qquad s=1,\ldots,Q_{\partial m}.
\]
The corresponding residuals are
\begin{equation*}
	\mathcal R^{\partial x}_{\Theta,k,r}
	:=
	D_xU_\Theta\bigl(
	t^{\partial x}_{k,r},
	x^{\partial x}_{k,r},
	m^{N,(k)}
	\bigr)
	\cdot n\bigl(x^{\partial x}_{k,r}\bigr),
\end{equation*}
and
\begin{equation*}
	\mathcal R^{\partial m}_{\Theta,k,s}
	:=
	D_mU_\Theta\bigl(
	t^{\partial m}_{k,s},
	x^{\partial m}_{k,s},
	m^{N,(k)},
	y^{\partial m}_{k,s}
	\bigr)
	\cdot n\bigl(y^{\partial m}_{k,s}\bigr).
\end{equation*}

The sampled residuals define the following mini-batch loss components:
\begin{equation*}
	\begin{aligned}
		\mathcal L_{\mathrm{pde}}(\Theta)
		&:=
		\frac1{M_{\mathrm{batch}}Q_{\mathrm{pde}}}
		\sum_{k=1}^{M_{\mathrm{batch}}}
		\sum_{q=1}^{Q_{\mathrm{pde}}}
		\left|\mathcal R^{\mathrm{pde}}_{\Theta,k,q}\right|^2, \\
		\mathcal L_T(\Theta)
		&:=
		\frac1{M_{\mathrm{batch}}Q_T}
		\sum_{k=1}^{M_{\mathrm{batch}}}
		\sum_{p=1}^{Q_T}
		\left|\mathcal R^T_{\Theta,k,p}\right|^2, \\
		\mathcal L_{\partial x}(\Theta)
		&:=
		\frac1{M_{\mathrm{batch}}Q_{\partial x}}
		\sum_{k=1}^{M_{\mathrm{batch}}}
		\sum_{r=1}^{Q_{\partial x}}
		\left|\mathcal R^{\partial x}_{\Theta,k,r}\right|^2, \\
		\mathcal L_{\partial m}(\Theta)
		&:=
		\frac1{M_{\mathrm{batch}}Q_{\partial m}}
		\sum_{k=1}^{M_{\mathrm{batch}}}
		\sum_{s=1}^{Q_{\partial m}}
		\left|\mathcal R^{\partial m}_{\Theta,k,s}\right|^2 .
	\end{aligned}
\end{equation*}
The total objective is
\begin{equation*}
	\mathcal L(\Theta)
	:=
	\lambda_{\mathrm{pde}}\mathcal L_{\mathrm{pde}}(\Theta)
	+
	\lambda_T\mathcal L_T(\Theta)
	+
	\lambda_{\partial x}\mathcal L_{\partial x}(\Theta)
	+
	\lambda_{\partial m}\mathcal L_{\partial m}(\Theta),
\end{equation*}
where the four loss weights are set to
\(\lambda_{\mathrm{pde}}=\lambda_T=\lambda_{\partial x}
=\lambda_{\partial m}=1\) in all numerical experiments. The network is trained by minimizing \(\mathcal L(\Theta)\) with
respect to all parameters in \(\Theta=(\theta,\eta,b)\), and the resulting
optimized parameters are denoted by \(\Theta^\star\).

This grouped construction avoids rebuilding an independent pooled embedding for
every collocation point. Instead, several collocation points share the same
empirical-measure representation, while the loss still samples multiple
measures through the index \(k=1,\ldots,M_{\mathrm{batch}}\). Since the residual
requires repeated evaluations of \(U_\Theta\), \(D_xU_\Theta\),
\(D_mU_\Theta\), and \(\operatorname{div}_y[D_mU_\Theta]\), this grouping
reduces unnecessary computation and stabilizes the mini-batch estimate of the
training objective. The derivative construction and residual assembly are summarized in Figure~\ref{fig:stream}(b).

In the numerical experiments, the parameters are optimized in three stages:
\(4000\) Adam iterations with learning rate \(10^{-3}\), followed by \(1000\)
Adam cool-down iterations with learning rate \(10^{-4}\), and finally \(500\)
L-BFGS iterations with learning rate \(0.1\).

\begin{algorithm}[htbp]
	\caption{MCOL training algorithm}
	\label{alg:grouped_training}
	\begin{algorithmic}[1]
		\Require Training densities or density sampler; particle number \(N\);
		mini-batch size \(M_{\rm batch}\); collocation numbers
		\(Q_{\rm pde},Q_T,Q_{\partial x},Q_{\partial m}\);
		loss weights; optimizer schedule.
		\Ensure Trained MCOL parameters \(\Theta\).
		
		\State Initialize \(\Theta=(\theta,\eta,b)\).
		\State Generate empirical measures
		\[
		m^{N,(j)}
		=
		\frac1N\sum_{i=1}^{N}\delta_{\xi_i^{(j)}},
		\qquad j=1,\ldots,10000.
		\]
		
		\For{each optimization stage}
		\For{each training iteration}
		\State Draw a mini-batch
		\(\mathcal B\subset\{1,\ldots,10000\}\),
		with \(|\mathcal B|=M_{\rm batch}\).
		
		\State Generate collocation points for the interior, terminal,
		and boundary residuals.
		
		\State For each \(j\in\mathcal B\), evaluate
		\(U_\Theta(t,x,m^{N,(j)})\) using the branch--trunk representation.
		
		\State Compute the derivatives \(\partial_tU_\Theta, D_xU_\Theta, \Delta_xU_\Theta\) by automatic differentiation.
		
		\State Compute \(D_mU_\Theta\) and
		\(\operatorname{div}_y[D_mU_\Theta]\) from the MCOL architecture.
		
		\State Approximate the nonlocal terms by empirical quadrature:
		\[
		\begin{aligned}
			I_{1,\Theta}
			&=
			\frac1N\sum_{i=1}^{N}
			\operatorname{div}_y[D_mU_\Theta]
			(t,x,m^{N,(j)},\xi_i^{(j)}),
			\\
			I_{2,\Theta}
			&=
			\frac1N\sum_{i=1}^{N}
			D_mU_\Theta(t,x,m^{N,(j)},\xi_i^{(j)})
			\cdot
			D_pH\!\left(
			\xi_i^{(j)},
			D_xU_\Theta(t,\xi_i^{(j)},m^{N,(j)})
			\right).
		\end{aligned}
		\]
		
		\State Assemble the PDE, terminal, and boundary losses:
		\[
		\mathcal L(\Theta)
		=
		\lambda_{\rm pde}\mathcal L_{\rm pde}
		+
		\lambda_T\mathcal L_T
		+
		\lambda_{\partial x}\mathcal L_{\partial x}
		+
		\lambda_{\partial m}\mathcal L_{\partial m}.
		\]
		
		\State Update \(\Theta\) using the optimizer of the current stage.
		\EndFor
		\EndFor
		
		\State \Return \(\Theta\).
	\end{algorithmic}
\end{algorithm}

\subsection{Error components and evaluation protocol}
\label{subsec:error_decomposition}

The MCOL approximation is evaluated at empirical measures \(m^N\), whereas the
continuous reference solution, when available, is associated with an underlying
measure \(m\), usually induced by a density \(\rho\). To separate the neural
approximation error from the error introduced by the empirical representation of
the measure, we add and subtract \(U(t,x,m^N)\). For fixed \((t,x)\), this gives
\begin{equation}
	\label{eq:error_decomposition_method}
	\begin{aligned}
		\underbrace{U_\Theta(t,x,m^N)-U(t,x,m)}_{e_{\mathrm{tot}}(t,x;m^N)}
		&=
		\underbrace{U_\Theta(t,x,m^N)-U(t,x,m^N)}_{e_{\mathrm{net}}(t,x;m^N)}
		\\
		&\quad+
		\underbrace{U(t,x,m^N)-U(t,x,m)}_{e_{\mathrm{disc}}(t,x;m^N,m)} .
	\end{aligned}
\end{equation}
Here \(e_{\mathrm{net}}\) is the neural approximation error at the empirical
measure, while \(e_{\mathrm{disc}}\) is the measure-discretization error caused
by replacing \(m\) with \(m^N\).

This decomposition is consistent with the usual stability viewpoint for
functions defined on probability measures. If, for fixed \((t,x)\), the map
\(m\mapsto U(t,x,m)\) is Lipschitz continuous with respect to the Wasserstein
distance \(W_1\), then
\[
\left|U(t,x,m^N)-U(t,x,m)\right|
\le
C(t,x) W_1(m^N,m).
\]
Such an estimate follows, for instance, when \(U\) is differentiable with
respect to the measure variable and its intrinsic derivative is uniformly
bounded on the bounded domain \(\Omega\)
\cite{Cardaliaguet2019,CarmonaDelarue2018I}. Thus \(e_{\mathrm{disc}}\)
reflects the accuracy of the empirical approximation of the measure argument,
whereas \(e_{\mathrm{net}}\) reflects the approximation capacity and training
accuracy of the neural representation on the same empirical input.

By the triangle inequality,
\begin{equation*}
	\|e_{\mathrm{tot}}\|_{L^2}
	\le
	\|e_{\mathrm{net}}\|_{L^2}
	+
	\|e_{\mathrm{disc}}\|_{L^2}.
\end{equation*}
The corresponding relative errors are defined as
\begin{equation}
	\label{eq:relative_error_components}
	E_{\mathrm{tot}}
	:=
	\frac{\|U_\Theta(t,x,m^N)-U(t,x,m)\|_{L^2}}
	{\|U(t,x,m)\|_{L^2}},
	\qquad
	E_{\mathrm{net}}
	:=
	\frac{\|e_{\mathrm{net}}\|_{L^2}}
	{\|U(t,x,m)\|_{L^2}},
	\qquad
	E_{\mathrm{disc}}
	:=
	\frac{\|e_{\mathrm{disc}}\|_{L^2}}
	{\|U(t,x,m)\|_{L^2}}.
\end{equation}
All \(L^2\) norms are approximated on the prescribed evaluation grid.

In Subsections~\ref{subsec:1d_benchmark} and~\ref{subsec:2d}, the test data are
generated independently of the training set. We sample \(1000\) independent test
densities using the same GRF-based procedure as in the training stage, but with
independent random realizations. For each test density, the evaluation points
cover the whole space--time domain with mesh size \(0.01\) in the temporal and
spatial directions. The reported mean total relative error is
\begin{equation*}
	\overline{E}_{\mathrm{tot}}
	=
	\frac{1}{1000}
	\sum_{j=1}^{1000}
	\left(
	\frac{
		\sum_i
		\left|
		U_\Theta(t_i,x_i,m^{N,(j)})
		-
		U(t_i,x_i,m^{(j)})
		\right|^2
	}{
		\sum_i
		\left|
		U(t_i,x_i,m^{(j)})
		\right|^2
	}
	\right)^{1/2},
\end{equation*}
where the summation over \(i\) is taken over all selected space--time evaluation
points for the \(j\)-th test measure. The mean errors
\(\overline{E}_{\mathrm{net}}\) and \(\overline{E}_{\mathrm{disc}}\) are
computed in the same way, with the numerators replaced by the corresponding
error terms in \eqref{eq:relative_error_components}.

\section{Numerical experiments}
\label{sec:numerical_experiments}

In this section, we evaluate the proposed MCOL method on several master equations. The experiments include 1D and 2D state-space benchmarks, an MFG-trajectory validation on \(\mathbb T^1\), and a systemic-risk problem with common noise. These examples are used to assess the approximation of the value function \(U\), the induced intrinsic derivative \(D_mU\), and the effect of empirical discretization of the measure argument. 

For the network architecture, the particle feature map 
\(\phi_\eta:\Omega\to\R^h\) is parameterized by a fully connected network with
output dimension \(h=100\), hidden width \(100\), and depth \(3\). The
aggregation map \(\psi_\eta:\R^h\to\R^q\) has the same hidden width and depth,
with output dimension \(q=50\). The trunk network
\(T_\theta:[0,T]\times\Omega\to\R^q\) takes the space--time variable \((t,x)\)
as input and is parameterized by a fully connected network with hidden width
\(100\), depth \(3\), and output dimension \(q=50\). The activation function is chosen as \(\tanh\). All numerical experiments were trained on a NVIDIA A100 GPU with 80~GB memory.

Unless otherwise stated in the corresponding comparison experiments, the
collocation and particle settings are chosen as follows. In
Subsection~\ref{subsec:1d_benchmark}, each training iteration uses
\(Q_{\mathrm{pde}}=200\), \(Q_T=20\), and
\(Q_{\partial x}=Q_{\partial m}=20\) per empirical measure; the number of
empirical measures in each training batch is \(M_{\mathrm{batch}}=20\), and both
training and testing empirical measures are represented by \(N_{\mathrm{train}}=N_{\mathrm{test}}=64\) particles. Subsections~\ref{subsec:characteristic_relation} and \ref{subsec:systemic_risk_common_noise} use the same particle number
and mini-batch size, with no boundary residual during training. In
Subsection~\ref{subsec:2d}, each training iteration uses
\(Q_{\mathrm{pde}}=500\), \(Q_T=50\), and
\(Q_{\partial x}=50\), with \(M_{\mathrm{batch}}=100\); the
training empirical measures use \(N_{\mathrm{train}}=64\) particles, whereas the testing empirical measures use \(N_{\mathrm{test}}=256\) particles.

\subsection{A one-dimensional state-space master equation}
\label{subsec:1d_benchmark}

We first consider a one-dimensional problem in \(
(0,1)\times [0,1]\times \PP([0,1]).\) This problem is obtained by specializing the master equation
\eqref{eq:master_general_polished} to the one-dimensional domain
\(\Omega=[0,1]\), with the homogeneous Neumann condition in the state variable
and the compatible no-flux condition in the measure variable \(y\) as described
in Section~\ref{subsec:master_equation}. The Hamiltonian is chosen as \(
H(x,p)=|p|^2/2.
\) The running cost \(F\) and terminal cost \(G\) are manufactured so that the exact solution is
\begin{equation*}
	U(t,x,m)
	=
	(T-t)+\cos(\pi x)\int_0^1 \cos(\pi z)\,m(dz).
\end{equation*}

We first validate the effectiveness of MCOL by comparing it with the PINN baseline on both in-distribution data (IDD) and out-of-distribution (OOD) cases. The comparison is conducted under the same training conditions and with comparable numbers of network parameters. Figure~\ref{fig:grf_density_particles_vertical} displays representative GRF densities \(\rho\) and the corresponding empirical particles used to construct \(m^N\) with \(N=64\). The IDD case uses the length-scale parameter \(l=0.2\). The OOD case uses a smaller length-scale parameter \(l=0.02\), leading to sharper peaks and stronger particle clustering. This setting tests whether the learned operator remains stable under changes in the smoothness of the input measure.

\begin{figure}[htbp]
	\centering
	
	\begin{subfigure}[t]{0.48\textwidth}
		\centering
		\includegraphics[width=\textwidth]{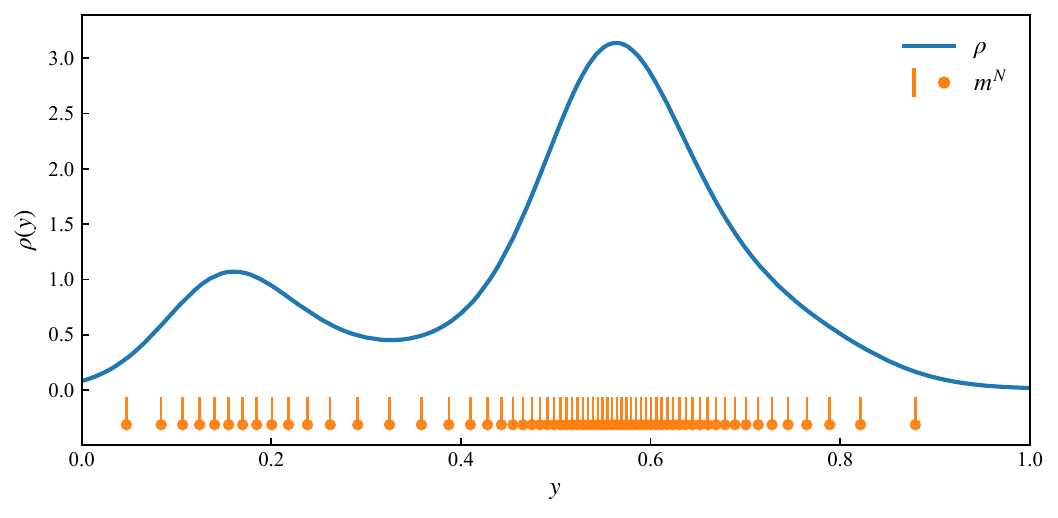}
		\caption{IDD.}
		\label{fig:saved_grf_density_particles}
	\end{subfigure}
	\hfill
	\begin{subfigure}[t]{0.48\textwidth}
		\centering
		\includegraphics[width=\textwidth]{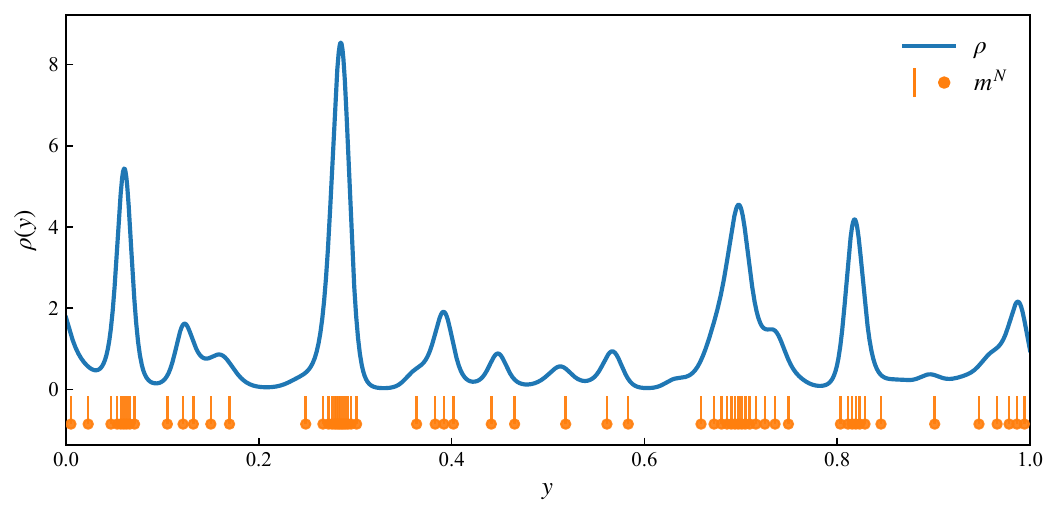}
		\caption{OOD.}
		\label{fig:ood_saved_grf_density_particles}
	\end{subfigure}
	
	\caption{Examples of GRF densities $\rho$ and the associated empirical particles defining $m^N$.}
	\label{fig:grf_density_particles_vertical}
\end{figure}

Figures~\ref{fig:three_way_error_heatmaps_columns} and
\ref{fig:three_way_error_heatmaps_columns1} compare the three pointwise error
components defined in \eqref{eq:error_decomposition_method} for the PINN baseline
and the proposed MCOL method, using the empirical measures shown in
Figure~\ref{fig:grf_density_particles_vertical} as inputs. In both the IDD and
OOD cases, the top rows show that the PINN baseline has
total errors \(e_{\mathrm{tot}}\) dominated by the network-induced component
\(e_{\mathrm{net}}\), while the empirical discretization error \(e_{\mathrm{disc}}\)
is several orders of magnitude smaller. The bottom rows show that MCOL
substantially reduces both \(e_{\mathrm{tot}}\) and \(e_{\mathrm{net}}\) over the
whole space--time domain. This improvement is already visible in the IDD case
and becomes more pronounced in the OOD case, where the GRF density is generated
with a much smaller length-scale parameter and contains sharper local
structures. Although the OOD setting increases the difficulty of the
approximation, MCOL keeps the network error at a much lower level than the
baseline. These results indicate that the
measure-consistent construction of the intrinsic derivative improves the
stability of the learned operator with respect to changes in the regularity of
the input distribution.

\begin{figure}[htbp]
	\centering
	\begin{subfigure}[t]{0.32\textwidth}
		\centering
		\includegraphics[width=\textwidth]{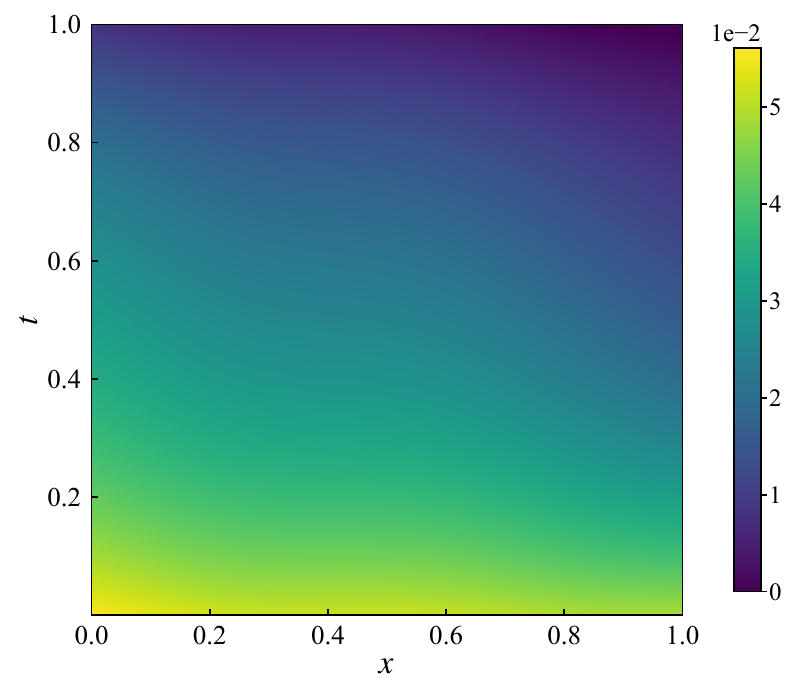}
		
		\vspace{0.5em}
		
		\includegraphics[width=\textwidth]{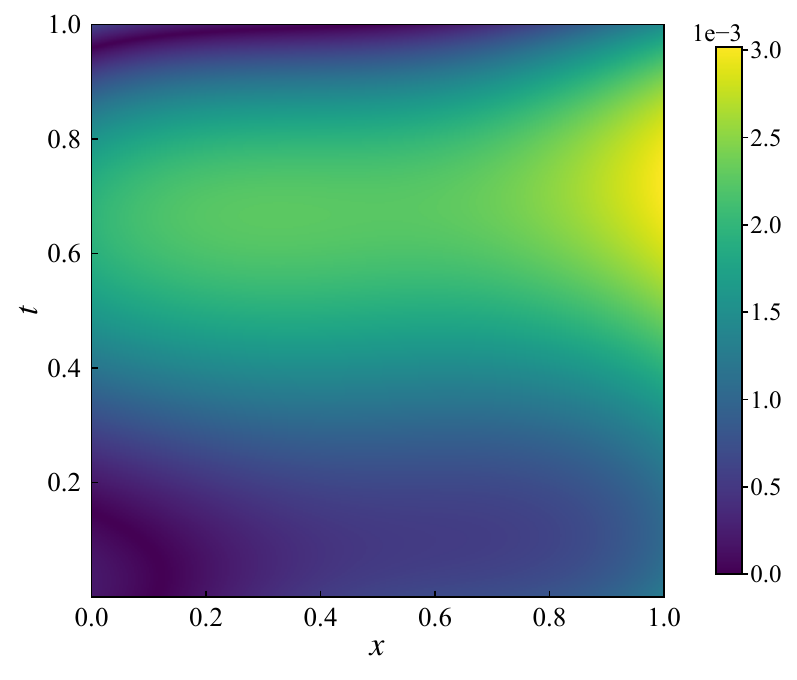}
		
		\caption{$|e_{\mathrm{tot}}|$.}
		\label{fig:err_net_cont_column}
	\end{subfigure}
	\hfill
	\begin{subfigure}[t]{0.32\textwidth}
		\centering
		\includegraphics[width=\textwidth]{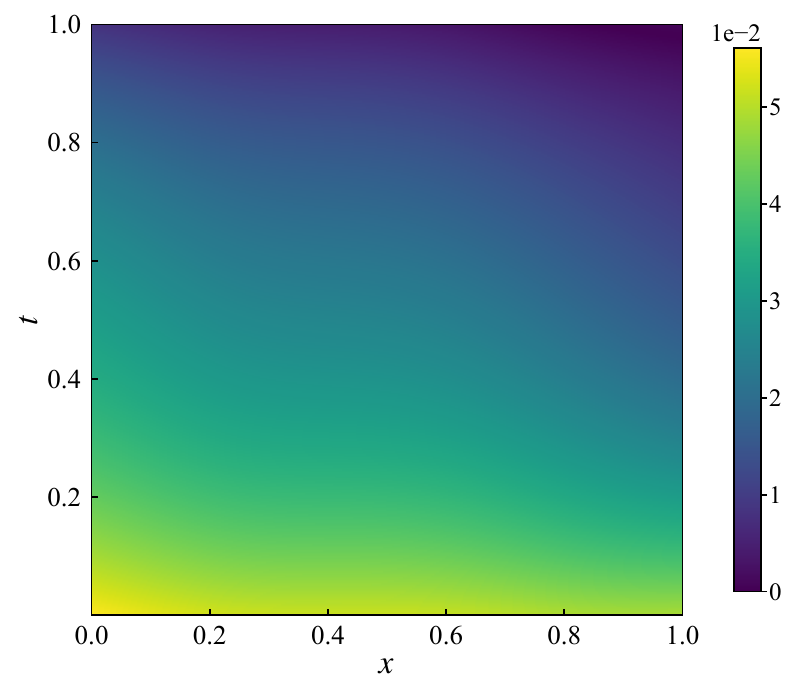}
		
		\vspace{0.5em}
		
		\includegraphics[width=\textwidth]{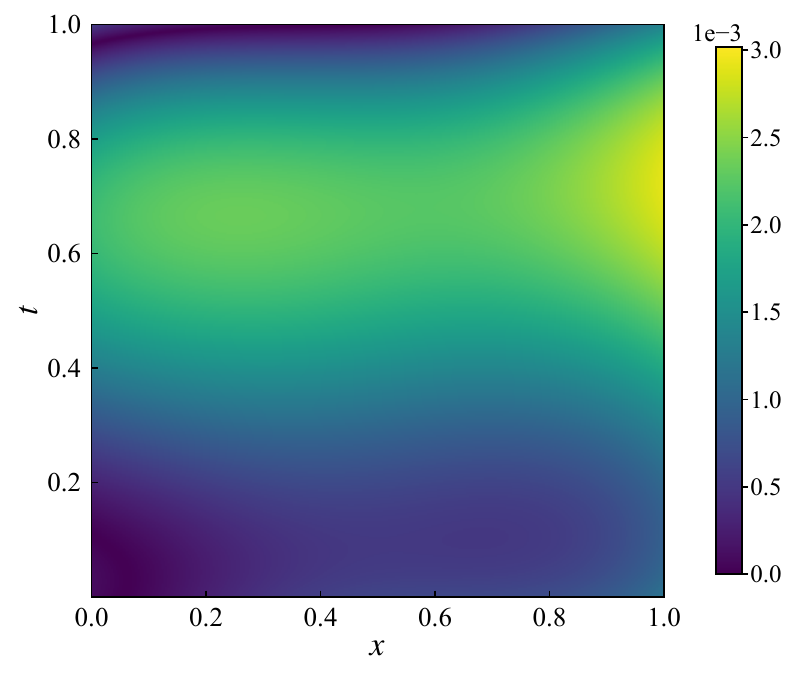}
		
		\caption{$|e_{\mathrm{net}}|$.}
		\label{fig:err_net_emp_column}
	\end{subfigure}
	\hfill
	\begin{subfigure}[t]{0.32\textwidth}
		\centering
		\includegraphics[width=\textwidth]{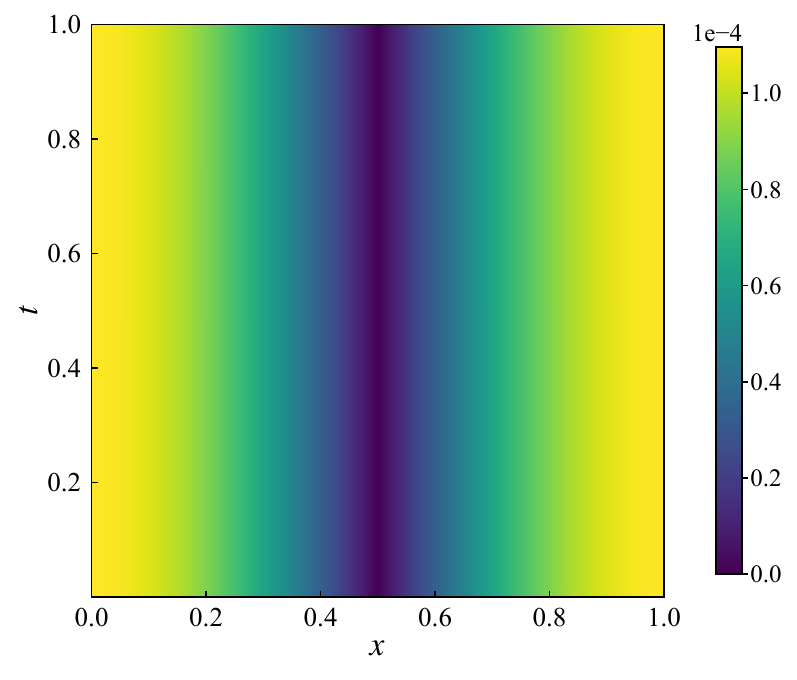}
		
		\vspace{0.5em}
		
		\includegraphics[width=\textwidth]{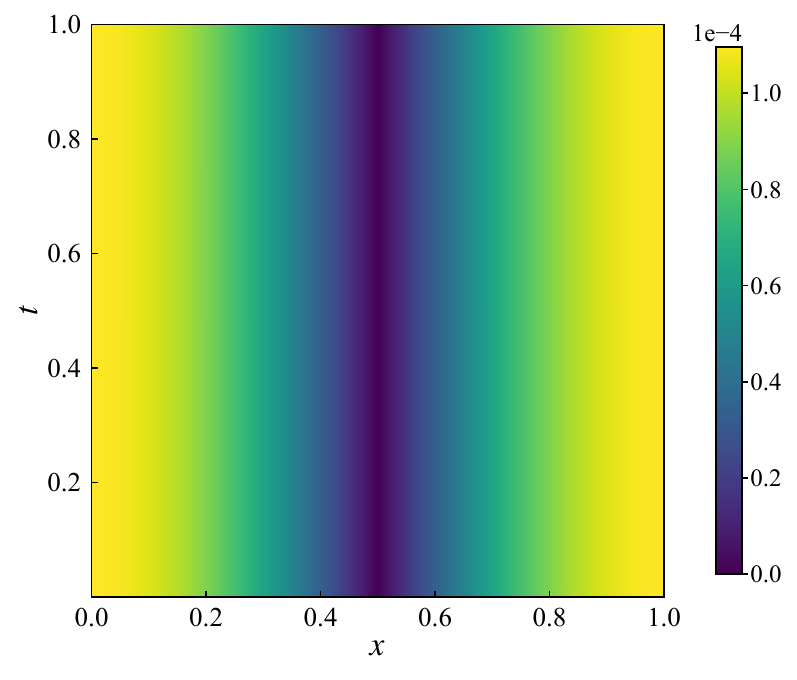}
		
		\caption{$|e_{\mathrm{disc}}|$.}
		\label{fig:err_emp_cont_column}
	\end{subfigure}
	
	\caption{
		IDD pointwise error comparison.
		Top: PINN baseline; bottom: MCOL.
	}
	\label{fig:three_way_error_heatmaps_columns}
\end{figure}

\begin{figure}[htbp]
	\centering
	\begin{subfigure}[t]{0.32\textwidth}
		\centering
		\includegraphics[width=\textwidth]{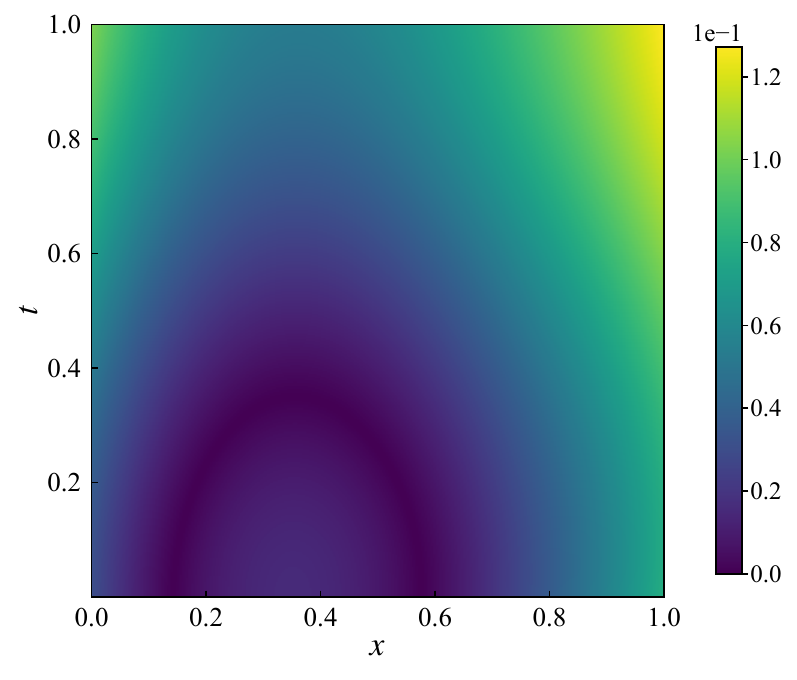}
		
		\vspace{0.5em}
		
		\includegraphics[width=\textwidth]{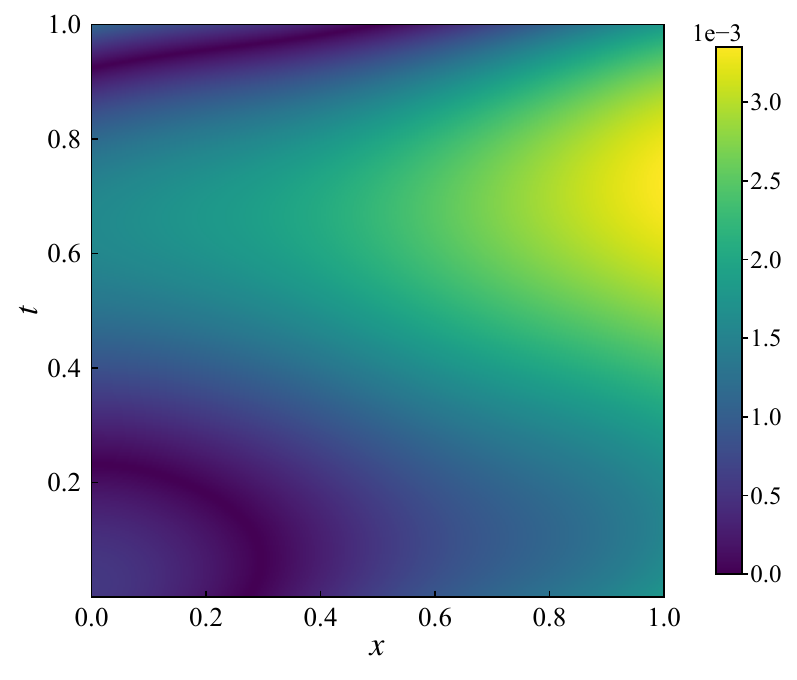}
		
		\caption{$|e_{\mathrm{tot}}|$.}
		\label{fig:err_net_cont_column1}
	\end{subfigure}
	\hfill
	\begin{subfigure}[t]{0.32\textwidth}
		\centering
		\includegraphics[width=\textwidth]{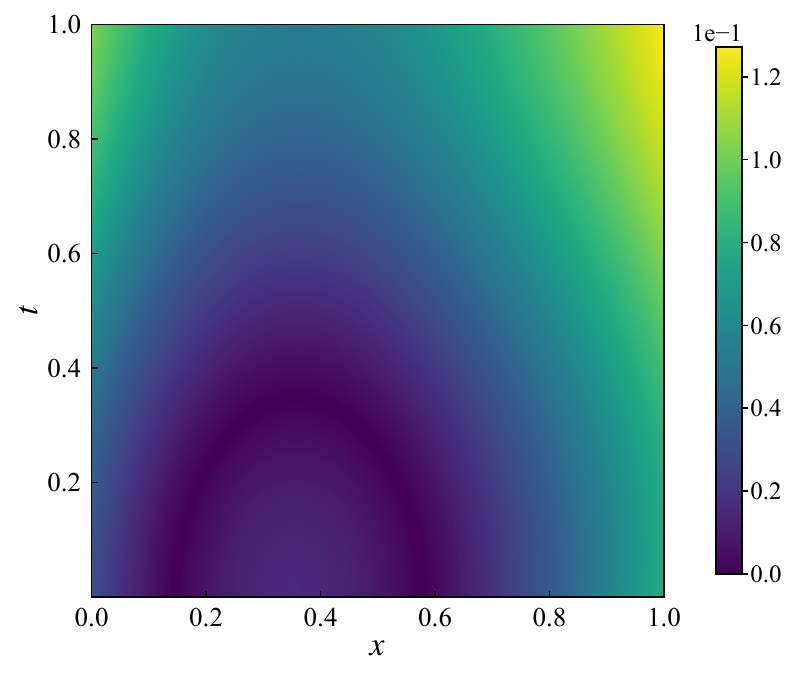}
		
		\vspace{0.5em}
		
		\includegraphics[width=\textwidth]{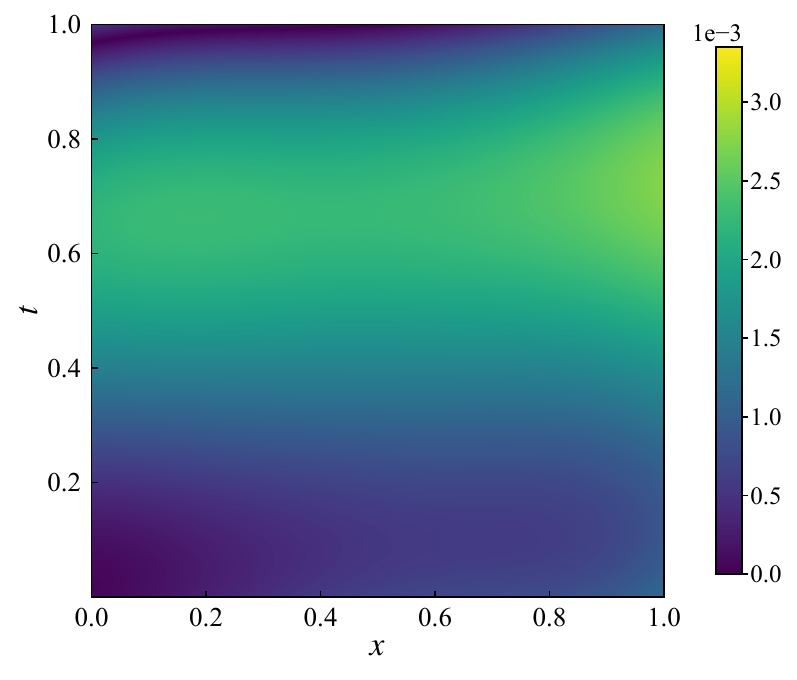}
		
		\caption{$|e_{\mathrm{net}}|$.}
		\label{fig:err_net_emp_column1}
	\end{subfigure}
	\hfill
	\begin{subfigure}[t]{0.31\textwidth}
		\centering
		\includegraphics[width=\textwidth]{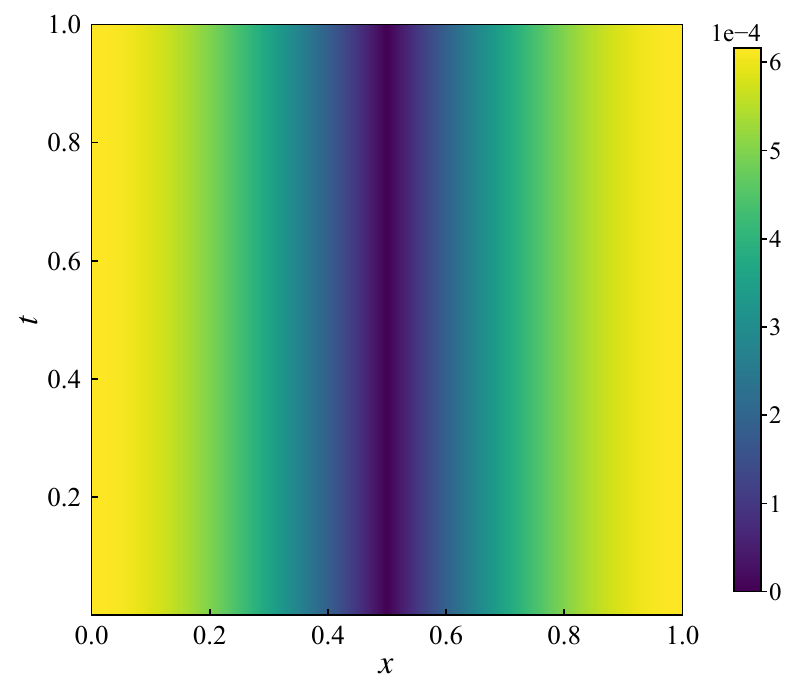}
		
		\vspace{0.5em}
		
		\includegraphics[width=\textwidth]{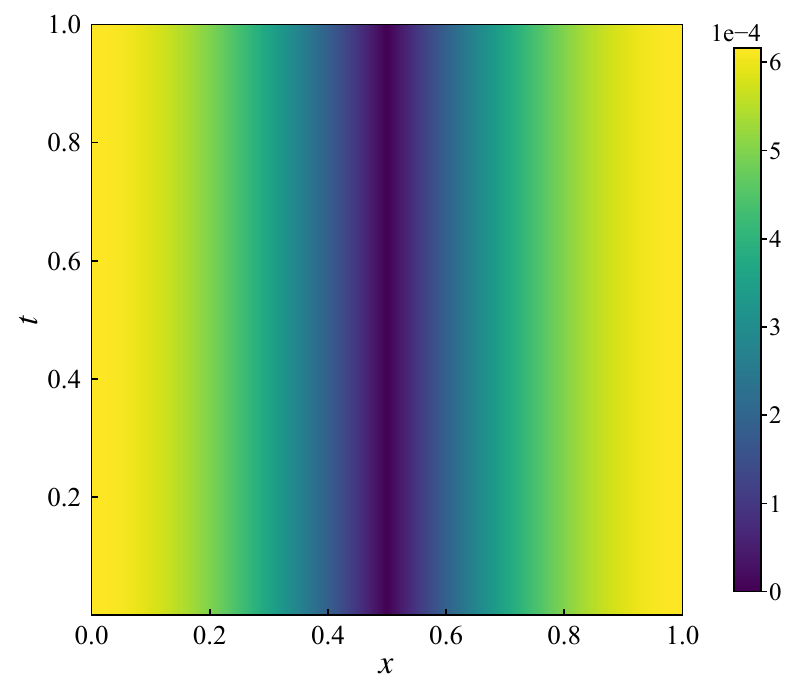}
		
		\caption{$|e_{\mathrm{disc}}|$.}
		\label{fig:err_emp_cont_column1}
	\end{subfigure}
	
	\caption{
		OOD pointwise error comparison.
		Top: PINN baseline; bottom: MCOL method.
	}
	\label{fig:three_way_error_heatmaps_columns1}
\end{figure}

To evaluate the performance over multiple input measures, we report the mean \(L^2\) relative errors and the corresponding standard deviations over 1000
independently generated test measures. Figure~\ref{fig:benchmark_error_plots}
shows the dependence of the three error components on the particle number \(N\).
For both IDD and OOD test measures generated from GRF densities,
\(E_{\mathrm{disc}}\) decreases as \(N\) increases, confirming that this term
mainly measures the empirical discretization error in the measure argument. The
decay is faster in the IDD case, where the densities are smoother and can be
resolved more efficiently by empirical particles. By contrast, the OOD densities
contain sharper local structures, which leads to a slower decay of
\(E_{\mathrm{disc}}\). The total error \(E_{\mathrm{tot}}\) decreases for small
\(N\), but then saturates near the level of \(E_{\mathrm{net}}\). This behavior
is consistent with the decomposition in
\eqref{eq:error_decomposition_method}: once the empirical discretization error
becomes sufficiently small, the total error is dominated by the approximation
error of the trained neural operator. The same qualitative trend is observed in
the OOD case, with slightly larger errors due to the reduced regularity of the
input densities. These results indicate that increasing the particle resolution
improves accuracy up to the network-approximation limit.

\begin{figure}[htbp]
	\centering
	
	\begin{subfigure}[t]{0.48\textwidth}
		\centering
		\includegraphics[width=\textwidth]{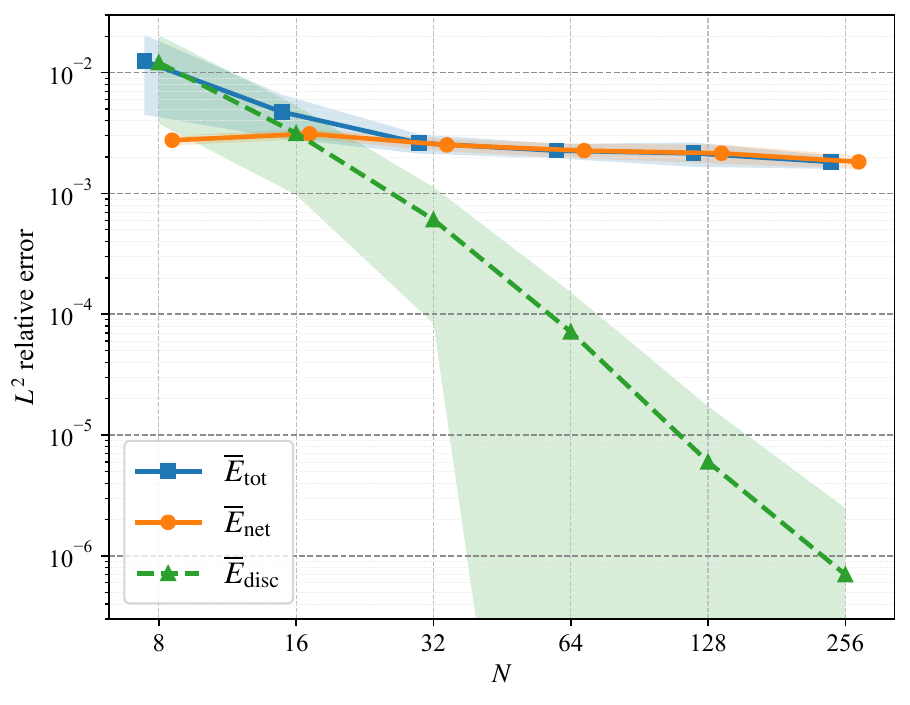}
		\caption{IDD.}
		\label{fig:idd_error_plot}
	\end{subfigure}
	\hfill
	\begin{subfigure}[t]{0.48\textwidth}
		\centering
		\includegraphics[width=\textwidth]{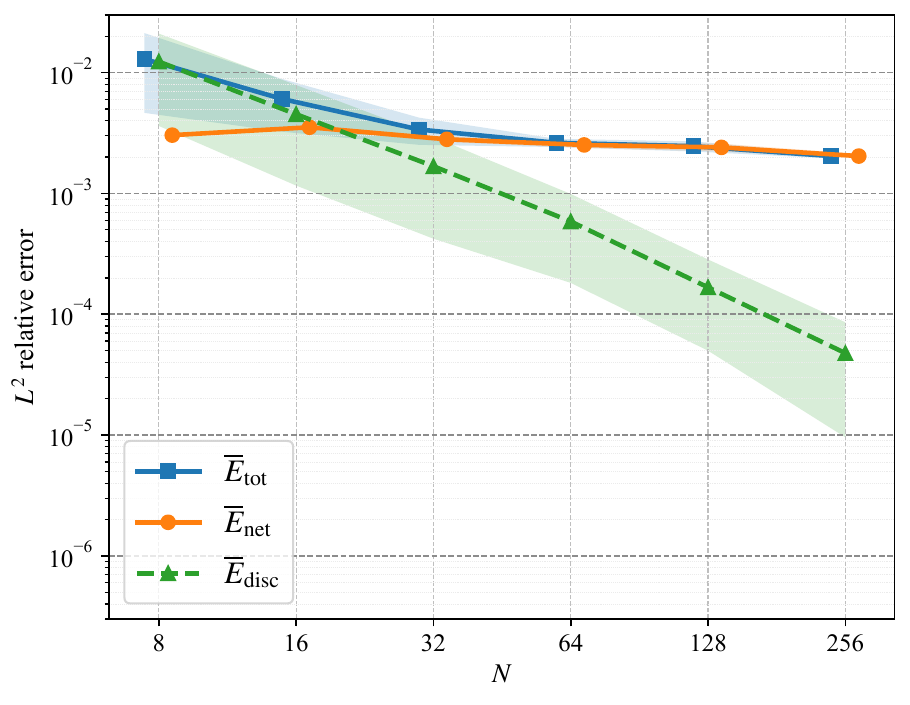}
		\caption{OOD.}
		\label{fig:ood_error_plot}
	\end{subfigure}
	
	\caption{
		Mean $L^2$ relative errors and corresponding standard deviations as functions of the particle number $N$.
	}
	\label{fig:benchmark_error_plots}
\end{figure}

To further assess the generalization of the trained network with respect to the
particle resolution, Figure~\ref{fig:dense_Ntest_heatmaps} reports the mean
network \(L^2\) relative error for different training and testing particle
numbers. For both IDD and OOD test densities, the error is mainly governed by
the training particle number \(N_{\mathrm{train}}\), whereas its dependence on
the testing particle number \(N_{\mathrm{test}}\) is relatively weak. This
behavior is desirable for an operator learning method on empirical measures,
since a stable measure representation learned from empirical inputs allows the
resulting operator to be evaluated at different particle resolutions without a
significant loss of accuracy. Increasing \(N_{\mathrm{train}}\) generally
improves the accuracy, because more training particles provide a more faithful
discretization of the measure-dependent terms in the residual. These results
indicate that the proposed MCOL architecture is stable with respect to changes in the particle resolution. Consequently, the network can be trained with a moderate number of particles to reduce computational cost, and then evaluated with a larger number of particles to further reduce the empirical discretization error.

\begin{figure}[htbp]
	\centering
	
	\begin{subfigure}[t]{0.47\textwidth}
		\centering
		\includegraphics[width=\textwidth]{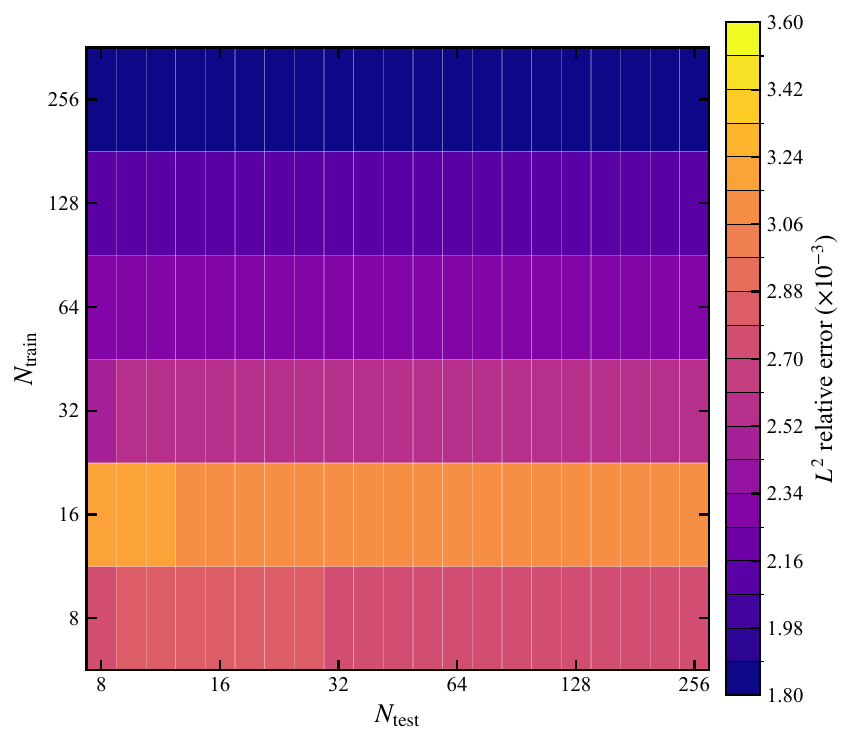}
		\caption{IDD.}
		\label{fig:idd_mean_dense_Ntest}
	\end{subfigure}
	\hfill
	\begin{subfigure}[t]{0.47\textwidth}
		\centering
		\includegraphics[width=\textwidth]{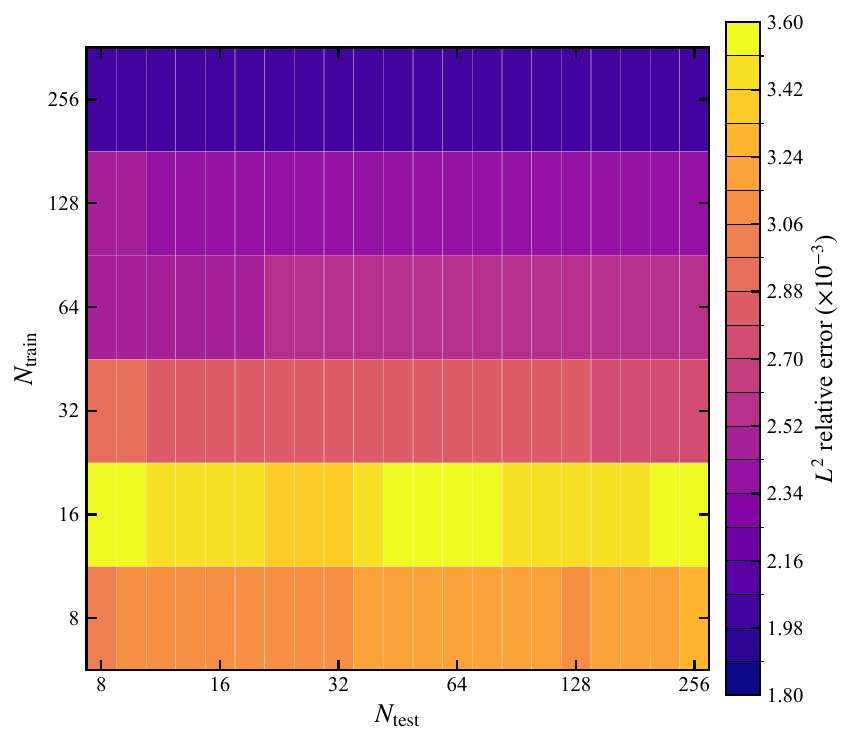}
		\caption{OOD.}
		\label{fig:ood_mean_dense_Ntest}
	\end{subfigure}
	
	\caption{
		Mean network $L^2$ relative error versus training and testing particle numbers.
	}
	\label{fig:dense_Ntest_heatmaps}
\end{figure}

\subsection{A two-dimensional state-space master equation}
\label{subsec:2d}

Let \(\Omega=[0,1]^2\). We consider the two-dimensional master equation in \((0,1)\times\Omega\times\PP(\Omega).\) The Hamiltonian is chosen as \(H(x,p)=|p|^2/2.\) 
In this example, exact-value Dirichlet boundary supervision is imposed on
\((0,T)\times\partial\Omega\times\PP(\Omega)\). The running cost \(F\), terminal cost \(G\), and boundary values are prescribed by the exact solution:
\begin{equation*}
	U(t,x,m)
	=
	t+|x|^2
	\left[
	\int_{\Omega}|z|^2\,dm(z)
	+
	\left(\int_{\Omega}z_1z_2\,dm(z)\right)^2
	+
	\sin\left(
	\int_{\Omega}(z_1^2z_2+z_1z_2^2)\,dm(z)
	\right)
	\right].
\end{equation*}

Tables~\ref{tab:2d_measure_pool_training} and
\ref{tab:2d_fixed_measure_training} compare two strategies for selecting
training measures in the two-dimensional experiment. In both cases, the
networks are trained with \(N_{\mathrm{train}}=64\) particles per empirical
measure and evaluated with \(N_{\mathrm{test}}=64\) and
\(N_{\mathrm{test}}=256\). In Table~\ref{tab:2d_measure_pool_training}, the
empirical measures are selected according to the grouped training procedure in
Subsection~\ref{subsec:residual_construction}: \(10\,000\) empirical measures
are generated in advance, and \(M_{\mathrm{batch}}\) of them are selected at
each iteration. In Table~\ref{tab:2d_fixed_measure_training}, only
\(M_{\mathrm{batch}}\) empirical measures are generated and reused throughout
training. The strategy in Table~\ref{tab:2d_measure_pool_training} gives
smaller errors for all tested values of \(M_{\mathrm{batch}}\), especially for
small and moderate batches. For example, when \(M_{\mathrm{batch}}=20\), it gives
\(\overline E_{\rm net}\approx 1.39\times10^{-2}\), compared with
\(4.39\times10^{-2}\) for the fixed-measure strategy; when
\(M_{\mathrm{batch}}=100\), \(\overline E_{\rm net}\) decreases to about
\(5.25\times10^{-3}\), while the fixed-measure strategy remains around
\(2.24\times10^{-2}\). Since the training times are comparable, this improvement
mainly reflects the greater diversity of training measures. The discretization
error \(E_{\rm disc}\) is not listed because it depends on
\(N_{\mathrm{test}}\), rather than on the sampling strategy or
\(M_{\mathrm{batch}}\). In this experiment,
\(\overline E_{\rm disc}=4.416\times10^{-3}\) for
\(N_{\mathrm{test}}=64\) and \(9.105\times10^{-4}\) for
\(N_{\mathrm{test}}=256\). Thus, increasing \(N_{\mathrm{test}}\) can further
reduce the total error once the neural approximation error becomes comparable
to the empirical discretization error. This is seen in
Table~\ref{tab:2d_measure_pool_training} with \(M_{\mathrm{batch}}=100\), where
increasing \(N_{\mathrm{test}}\) from \(64\) to \(256\) leaves
\(\overline E_{\rm net}\) almost unchanged but reduces
\(\overline E_{\rm tot}\) from \(7.064\times10^{-3}\) to
\(5.330\times10^{-3}\).

\begin{table}[htbp]
	\centering
	\caption{
		Relative errors for the two-dimensional master equation using \(M_{\mathrm{batch}}\) empirical measures sampled per iteration from \(10\,000\) training measures.
	}
	\label{tab:2d_measure_pool_training}
	\setlength{\tabcolsep}{5.5pt}
	\renewcommand{\arraystretch}{1.12}
	\footnotesize
	\begin{tabular}{c cc cc c}
		\toprule
		\multirow{2}{*}{$M_{\mathrm{batch}}$}
		& \multicolumn{2}{c}{$\overline E_{\rm tot}$}
		& \multicolumn{2}{c}{$\overline E_{\rm net}$}
		& \multirow{2}{*}{Time (h)}
		\\
		\cmidrule(lr){2-3}
		\cmidrule(lr){4-5}
		& $N_{\mathrm{test}}=64$ & $N_{\mathrm{test}}=256$
		& $N_{\mathrm{test}}=64$ & $N_{\mathrm{test}}=256$
		&
		\\
		\midrule
		1
		& $2.045\times 10^{-1}$ & $2.045\times 10^{-1}$
		& $2.044\times 10^{-1}$ & $2.045\times 10^{-1}$
		& 0.30
		\\
		10
		& $1.696\times 10^{-2}$ & $1.628\times 10^{-2}$
		& $1.629\times 10^{-2}$ & $1.627\times 10^{-2}$
		& 0.54
		\\
		20
		& $1.479\times 10^{-2}$ & $1.390\times 10^{-2}$
		& $1.390\times 10^{-2}$ & $1.387\times 10^{-2}$
		& 0.88
		\\
		50
		& $1.180\times 10^{-2}$ & $1.068\times 10^{-2}$
		& $1.067\times 10^{-2}$ & $1.065\times 10^{-2}$
		& 1.61
		\\
		100
		& $7.064\times 10^{-3}$ & $5.330\times 10^{-3}$
		& $5.265\times 10^{-3}$ & $5.248\times 10^{-3}$
		& 2.87
		\\
		\bottomrule
	\end{tabular}
\end{table}

\begin{table}[htbp]
	\centering
	\caption{
		Relative errors for the two-dimensional master equation trained on a fixed set of \(M_{\mathrm{batch}}\) empirical measures.
	}
	\label{tab:2d_fixed_measure_training}
	\setlength{\tabcolsep}{5.5pt}
	\renewcommand{\arraystretch}{1.12}
	\footnotesize
	\begin{tabular}{c cc cc c}
		\toprule
		\multirow{2}{*}{$M_{\mathrm{batch}}$}
		& \multicolumn{2}{c}{$\overline E_{\rm tot}$}
		& \multicolumn{2}{c}{$\overline E_{\rm net}$}
		& \multirow{2}{*}{Time (h)}
		\\
		\cmidrule(lr){2-3}
		\cmidrule(lr){4-5}
		& $N_{\mathrm{test}}=64$ & $N_{\mathrm{test}}=256$
		& $N_{\mathrm{test}}=64$ & $N_{\mathrm{test}}=256$
		&
		\\
		\midrule
		1
		& $3.879\times 10^{-1}$ & $3.878\times 10^{-1}$
		& $3.878\times 10^{-1}$ & $3.878\times 10^{-1}$
		& 0.28
		\\
		10
		& $8.641\times 10^{-2}$ & $8.616\times 10^{-2}$
		& $8.611\times 10^{-2}$ & $8.614\times 10^{-2}$
		& 0.57
		\\
		20
		& $4.429\times 10^{-2}$ & $4.397\times 10^{-2}$
		& $4.390\times 10^{-2}$ & $4.396\times 10^{-2}$
		& 0.88
		\\
		50
		& $3.572\times 10^{-2}$ & $3.536\times 10^{-2}$
		& $3.536\times 10^{-2}$ & $3.537\times 10^{-2}$
		& 1.70
		\\
		100
		& $2.301\times 10^{-2}$ & $2.240\times 10^{-2}$
		& $2.243\times 10^{-2}$ & $2.242\times 10^{-2}$
		& 2.82
		\\
		\bottomrule
	\end{tabular}
\end{table}

Figure~\ref{fig:m_2d} shows a two-dimensional GRF density
\(\rho\) on \(\Omega=[0,1]^2\) together with the empirical particles. The heatmap represents the continuous density,
while the orange points form the empirical measure \(m^N\) used in the
computation.

\begin{figure}[htbp]
	\centering
	\includegraphics[width=0.58\textwidth]{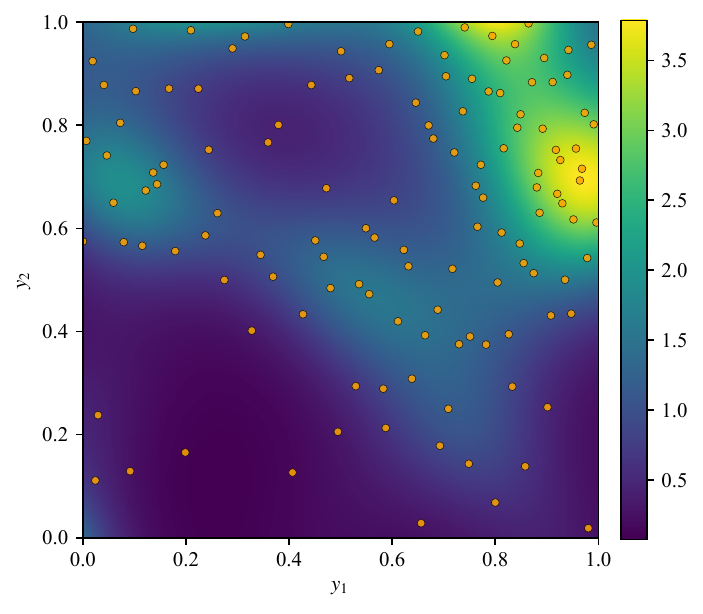}
	\caption{
		An example of a two-dimensional density $\rho$ with its associated empirical particles.
	}
	\label{fig:m_2d}
\end{figure}

We use the probability measure shown in Figure~\ref{fig:m_2d} to further evaluate the consistency of the network approximation to \(D_mU\). We evaluate \(D_mU(t,x,m,y)\) at
\((t,x_1,x_2)=(0.5,0.5,0.5)\) and plot its dependence on
\(y=(y_1,y_2)\in\Omega\). Since \(D_mU\) is a two-dimensional vector field, its
two components are displayed separately. Figure~\ref{fig:dm_slice_2d} compares
the exact intrinsic derivative, the network prediction, and the corresponding
pointwise absolute error. The predicted components accurately reproduce the
main spatial profiles of the exact derivative. The slice \(L^2\) relative errors are \(4.3059\times10^{-2}\) for \(D_{m,1}U\),
\(3.6746\times10^{-2}\) for \(D_{m,2}U\). These results
indicate that the error remains small relative to the magnitude of
\(D_mU\), showing that the learned operator captures not only the value
function but also its intrinsic measure derivative on this representative
two-dimensional test measure and its empirical approximation.

\begin{figure}[htbp]
	\centering
	
	% First row: D_{m,1}U
	\begin{subfigure}[t]{0.31\textwidth}
		\centering
		\includegraphics[width=\textwidth]{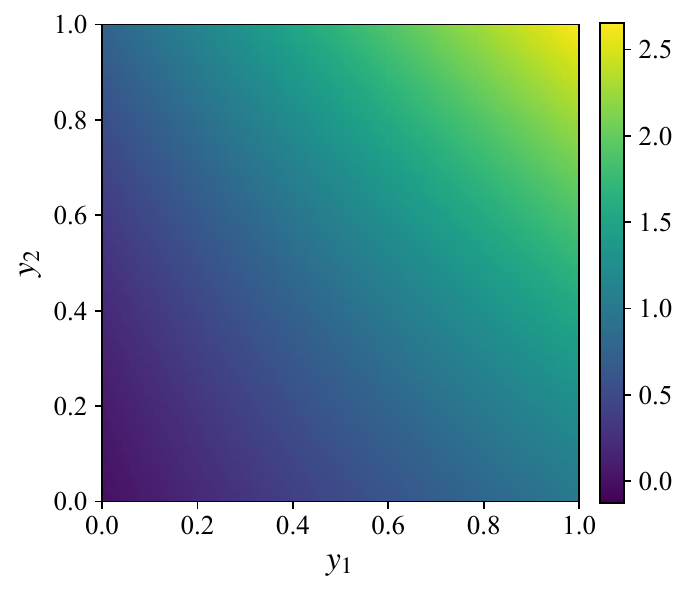}
	\end{subfigure}
	\hfill
	\begin{subfigure}[t]{0.31\textwidth}
		\centering
		\includegraphics[width=\textwidth]{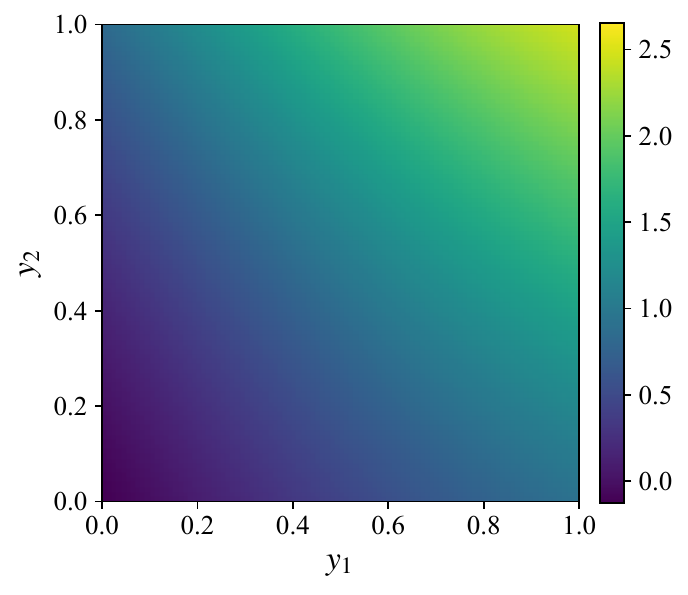}
	\end{subfigure}
	\hfill
	\begin{subfigure}[t]{0.32\textwidth}
		\centering
		\includegraphics[width=\textwidth]{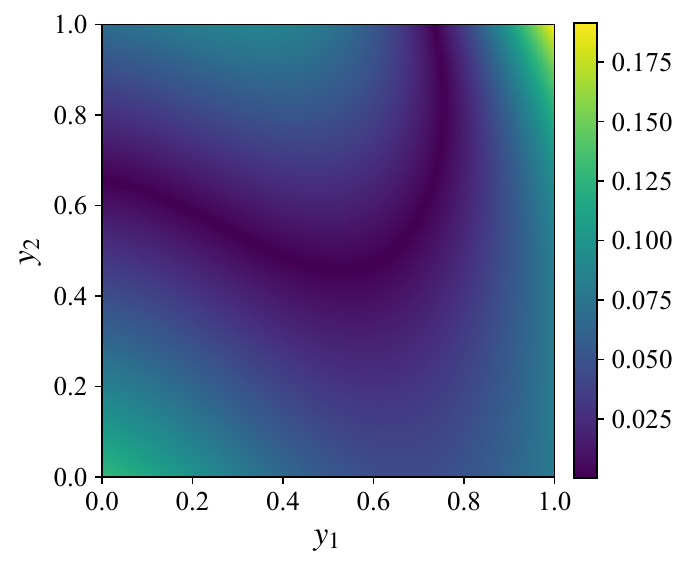}
	\end{subfigure}
	
	\vspace{0.6em}
	
	% Second row: D_{m,2}U
	\begin{subfigure}[t]{0.31\textwidth}
		\centering
		\includegraphics[width=\textwidth]{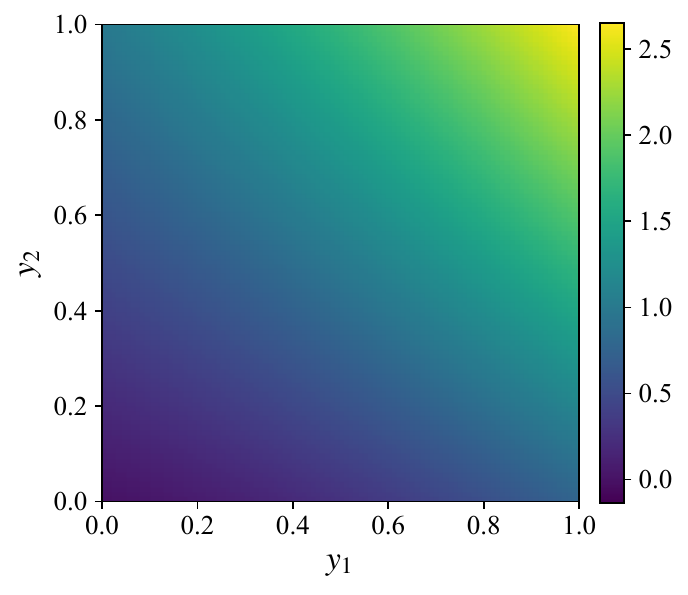}
		\caption{Exact $D_m$}
		\label{fig:dm2_exact}
	\end{subfigure}
	\hfill
	\begin{subfigure}[t]{0.31\textwidth}
		\centering
		\includegraphics[width=\textwidth]{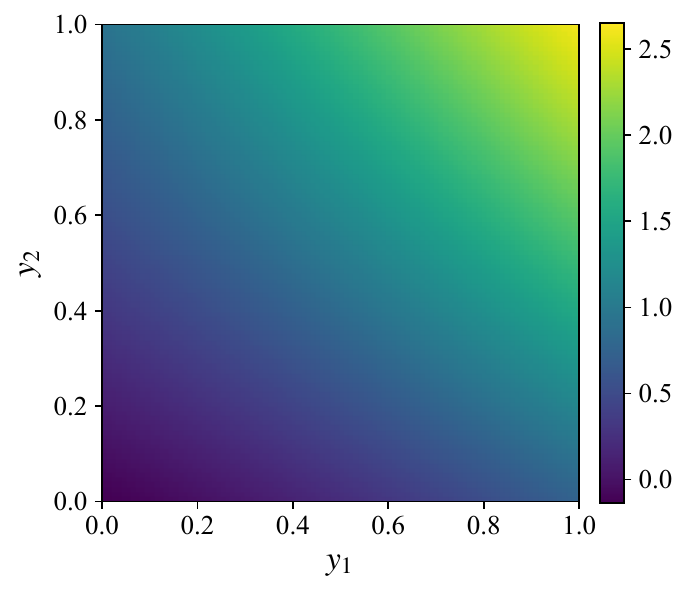}
		\caption{Predicted $D_m$}
		\label{fig:dm2_network}
	\end{subfigure}
	\hfill
	\begin{subfigure}[t]{0.31\textwidth}
		\centering
		\includegraphics[width=\textwidth]{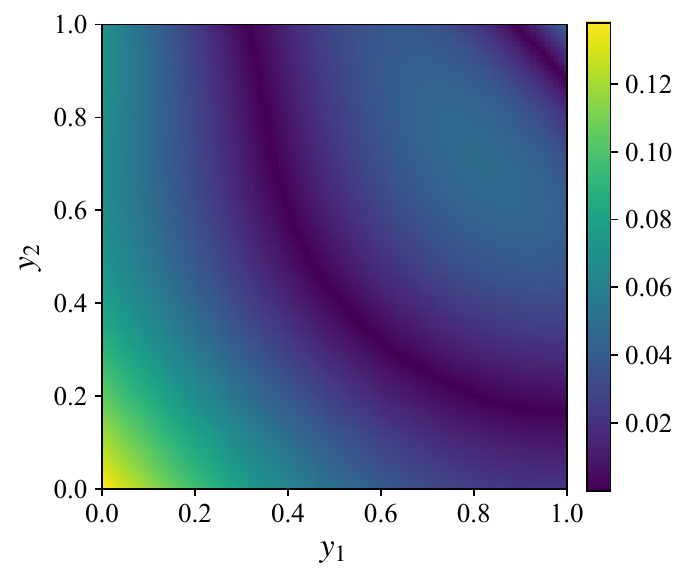}
		\caption{Absolute error}
		\label{fig:dm2_error}
	\end{subfigure}
	
	\caption{
		Comparison of the exact and predicted $D_mU$ at
		$(t,x_1,x_2)=(0.5,0.5,0.5)$.
		The first and second rows show $D_{m,1}U$ and $D_{m,2}U$, respectively.
	}
	\label{fig:dm_slice_2d}
\end{figure}

\subsection{Verification of the characteristic relation on \texorpdfstring{$\mathbb T^1$}{T1}}
\label{subsec:characteristic_relation}

We next verify the characteristic relation between the master equation and the
corresponding MFG system. Unlike the GRF-generated measures used during
training, the empirical measures in this test are constructed from a density
trajectory produced by an independently solved MFG system. This experiment
therefore assesses whether the learned operator can be evaluated along a
dynamically generated measure trajectory outside the GRF training distribution.
The test is carried out on the periodic domain
\(\mathbb T^1=\mathbb R/\mathbb Z\), with \(T=1\). In the implementation, the
torus is represented by the half-open interval \([0,1)\), so the endpoint is
not duplicated on the spatial grid. Periodicity is encoded by using the feature
map \(z\mapsto(\cos(2\pi z),\sin(2\pi z))\) for both the state variable and the
empirical particles.

We choose the quadratic Hamiltonian \(H(p)=p^2/2\). The running and terminal
costs are specified through a smooth nonlocal convolution coupling. For a
periodic kernel \(K_\sigma:\mathbb T^1\to\mathbb R\), define
\[
(K_\sigma*m)(x)
:=
\int_{\mathbb T^1}K_\sigma(x-z)\,m(\dd z).
\]
In the numerical tests, we take
\[
K_\sigma(x)
=
\exp\!\left(
-\frac{\sin^2(\pi x)}{2\sigma^2}
\right),
\qquad x\in\mathbb T^1,
\]
and set
\[
F(x,m)
=
a_F\cos(2\pi x)+\lambda_F(K_\sigma*m)(x),
\qquad
G(x,m)
=
a_G\cos(2\pi x)+\lambda_G(K_\sigma*m)(x),
\]
where \(a_F,a_G,\lambda_F,\lambda_G\in\mathbb R\) are prescribed constants.

For the density formulation of the MFG system, we use
\[
\begin{aligned}
	F_\rho(x,\rho)
	&:=
	a_F\cos(2\pi x)
	+
	\lambda_F\int_{\mathbb T^1}K_\sigma(x-z)\rho(z)\,\dd z,\\
	G_\rho(x,\rho)
	&:=
	a_G\cos(2\pi x)
	+
	\lambda_G\int_{\mathbb T^1}K_\sigma(x-z)\rho(z)\,\dd z .
\end{aligned}
\]
For a given initial density \(\rho_0\), the associated MFG system evolves a
value function \(u\) and a population density \(\rho\). It is written as
\begin{equation}
	\label{eq:MFG_system_T1_example}
	\left\{
	\begin{aligned}
		&-\partial_t u(t,x)-\nu\partial_{xx}u(t,x)
		+\frac12 |u_x(t,x)|^2
		=
		F_\rho\bigl(x,\rho(t,\cdot)\bigr),
		&& (t,x)\in [0,T)\times\mathbb T^1,\\
		&\partial_t \rho(t,x)-\nu\partial_{xx}\rho(t,x)
		-\partial_x\!\bigl(\rho(t,x)u_x(t,x)\bigr)
		=0,
		&& (t,x)\in (0,T]\times\mathbb T^1,\\
		&u(T,x)=G_\rho\bigl(x,\rho(T,\cdot)\bigr),
		&& x\in\mathbb T^1,\\
		&\rho(0,\cdot)=\rho_0,
		&& \text{on }\mathbb T^1.
	\end{aligned}
	\right.
\end{equation}
We use the initial density
\begin{equation*}
	\rho_0(x)=1+0.3\cos(2\pi x),
	\qquad x\in\mathbb T^1 .
\end{equation*}
The density \(\rho(t,\cdot)\) induces the time-dependent probability measure
\begin{equation}
	\label{eq:systemrho}
	m(t)(\dd z)=\rho(t,z)\,\dd z .
\end{equation}
This measure is used below to connect the density-based MFG system with the
master equation.

The corresponding master equation is posed in
\((0,T)\times\mathbb T^1\times\mathcal P(\mathbb T^1)\):
\begin{equation}
	\label{eq:master_equation_T1_example}
	\left\{
	\begin{aligned}
		&-\partial_t U(t,x,m)-\nu\partial_{xx}U(t,x,m)
		+\frac12\bigl|\partial_xU(t,x,m)\bigr|^2 \\
		&\quad
		-\nu\int_{\mathbb T^1}
		\partial_y\!\left[D_mU(t,x,m,y)\right]\,m(\dd y)
		+\int_{\mathbb T^1}
		D_mU(t,x,m,y)\,
		\partial_xU(t,y,m)\,m(\dd y)
		=
		F(x,m),\\
		&\hspace{9.2cm}
		(t,x,m)\in [0,T)\times\mathbb T^1\times\mathcal P(\mathbb T^1),\\
		&U(T,x,m)=G(x,m),
		\qquad
		(x,m)\in\mathbb T^1\times\mathcal P(\mathbb T^1).
	\end{aligned}
	\right.
\end{equation}
The coefficient \(\nu\) in the nonlocal diffusion term is the same as the
diffusion coefficient in the Fokker--Planck equation. The parameters are chosen as \(\nu=0.1\), \(\sigma=0.2\), \(a_F=0.6\), \(\lambda_F=0.5\), \(a_G=0.2\), \(\lambda_G=0.3\).

This problem provides a direct numerical check of the characteristic relation.
If \(U\) solves \eqref{eq:master_equation_T1_example} and \((u,\rho)\) solves
\eqref{eq:MFG_system_T1_example}, then the density \(\rho(t,\cdot)\) generated
by the MFG system induces the probability measure
\(m(t)(\dd z)=\rho(t,z)\,\dd z\), which serves as the measure argument of the
master equation along the MFG trajectory. Formally, this gives
\begin{equation}
	\label{eq:characteristic_relation_example}
	u(t,x)=U\bigl(t,x,m(t)\bigr),
	\qquad
	(t,x)\in[0,T]\times\mathbb T^1 .
\end{equation}
Thus, the density variable in the MFG system and the measure variable in the
master equation are linked through \eqref{eq:systemrho}. After training
\(U_\Theta\), we evaluate the learned operator along the measure
trajectory \(m(t)\) and compare it with the independently computed MFG value function.

More precisely, the MFG system \eqref{eq:MFG_system_T1_example} is solved by an
independent finite-difference solver \cite{Achdou2010} on
\([0,1]\times\mathbb T^1\), where \(\mathbb T^1\) is discretized as the
periodic interval \([0,1)\). We use a uniform grid with
\(\Delta t=\Delta x=10^{-3}\), giving the reference data
\[
\bigl\{
u^{\rm ref}(t_n,x_i),\,
\rho^{\rm ref}(t_n,x_i)
\bigr\}_{0\le n\le N_t,\;0\le i\le N_x-1},
\]
where \(t_n=n\Delta t\) and \(x_i=i\Delta x\). For each time level \(t_n\), the
reference density \(\rho^{\rm ref}(t_n,\cdot)\) defines the reference
probability measure
\[
m^{\rm ref}(t_n)(\dd z)
=
\rho^{\rm ref}(t_n,z)\,\dd z .
\]
We construct the network input \(m_n^N\) by first interpolating the grid-based
reference density \(\rho^{\rm ref}(t_n,\cdot)\) on \(\mathbb T^1\), and then
applying the ICDF sampling procedure described in
Subsection~\ref{subsec:continuous_to_empirical_measures}. This yields
\[
m_n^N
=
\frac1N\sum_{j=1}^N\delta_{\xi_j^{(n)}}
\approx
m^{\rm ref}(t_n),
\]
which is used as the network input through the particles
\(\{\xi_j^{(n)}\}_{j=1}^N\). The comparison is performed between
\(u^{\rm ref}(t_n,x_i)\) and \(U_\Theta(t_n,x_i,m_n^N)\) over the full
space--time grid.

Figure~\ref{fig:mfg_master_trajectory_comparison} verifies the characteristic relation along an independently computed MFG trajectory. The reference solution \(u^{\rm ref}(t,x)\) and the MCOL prediction \(U_\Theta(t,x,m_n^N)\) are shown in Figures~\ref{fig:mfg_master_trajectory_comparison}(a) and~\ref{fig:mfg_master_trajectory_comparison}(b), respectively. The two profiles agree well over the full space--time domain, including the time-dependent variation induced by the evolving measure \(m(t)\). Figure~\ref{fig:mfg_master_trajectory_comparison}(c) shows the reference density \(\rho^{\rm ref}(t,x)\), from which the empirical measures \(m_n^N\) are constructed at each time level. Although this density trajectory is generated by an independent MFG solver rather than by the GRF sampling procedure used in training, the learned operator remains accurate along the trajectory. The absolute error in Figure~\ref{fig:mfg_master_trajectory_comparison}(d) is small throughout most of the domain, with larger values mainly localized near the terminal time. The \(L^2\) relative error over the full space--time grid is \(1.78\times 10^{-2}\). These results confirm that the learned master-equation operator is consistent with \eqref{eq:characteristic_relation_example} when evaluated on a dynamically generated non-GRF measure trajectory.

\begin{figure}[htbp]
	\centering

	\begin{subfigure}[t]{0.46\textwidth}
		\centering
		\includegraphics[width=\textwidth]{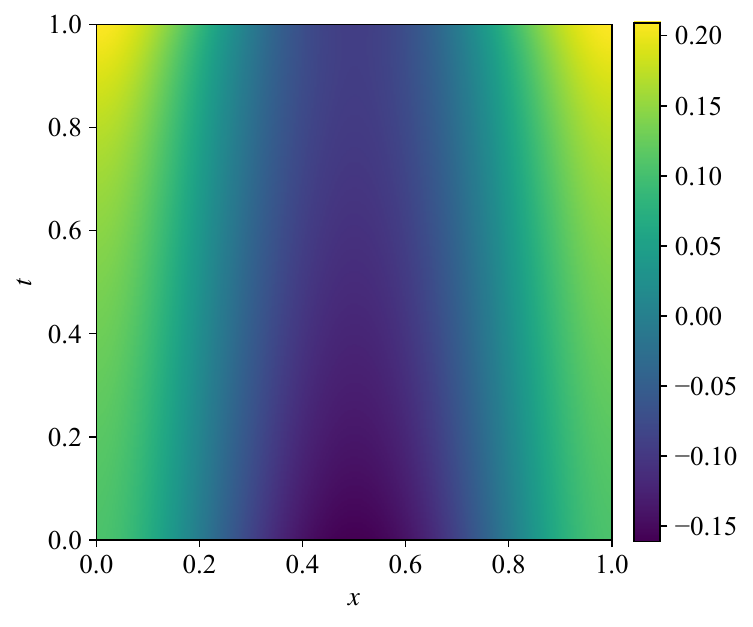}
		\caption{MFG system reference solution $u^{\mathrm{ref}}(t,x)$}
		\label{fig:mfg_reference_u}
	\end{subfigure}
	\hfill
	\begin{subfigure}[t]{0.46\textwidth}
		\centering
		\includegraphics[width=\textwidth]{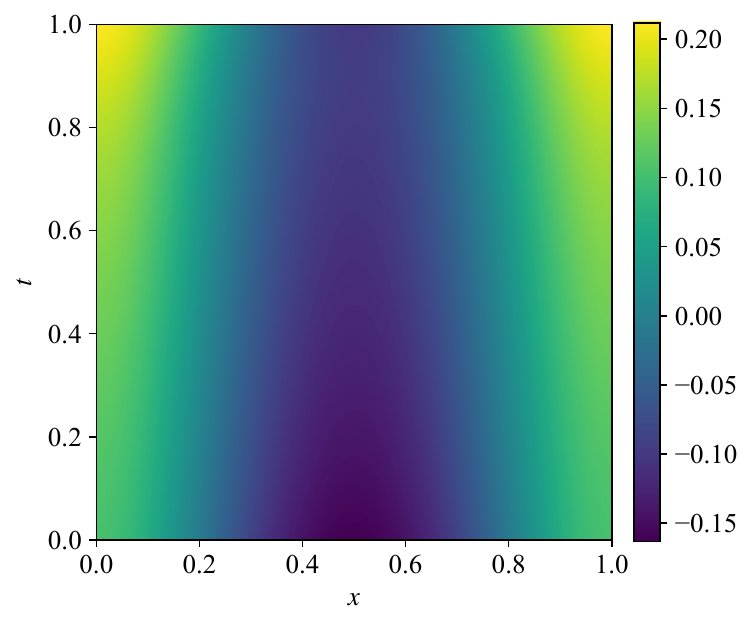}
		\caption{MCOL prediction $U_\Theta(t,x,m_n^N)$}
		\label{fig:master_network_prediction_u}
	\end{subfigure}
	
	\vspace{0.6em}
	\begin{subfigure}[t]{0.46\textwidth}
		\centering
		\includegraphics[width=\textwidth]{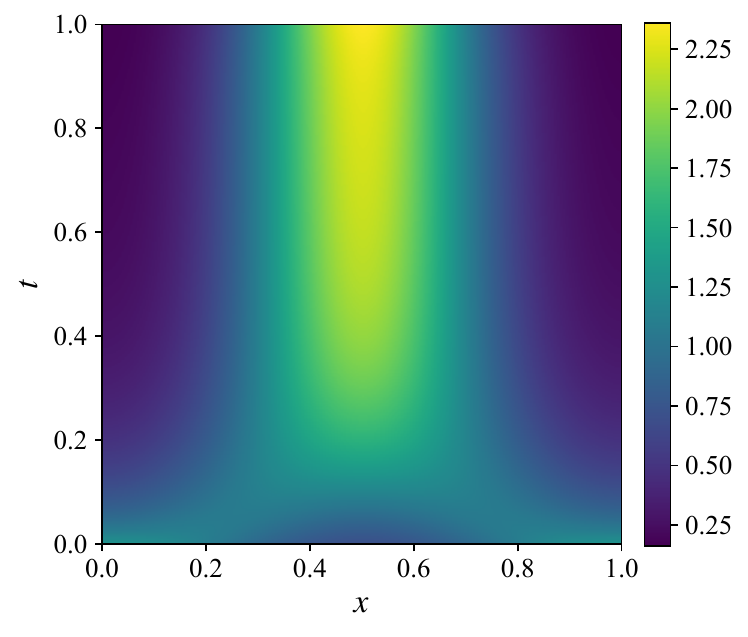}
		\caption{Reference density $\rho^{\mathrm{ref}}(t,x)$}
		\label{fig:mfg_density_ref}
	\end{subfigure}
	\hfill
	\begin{subfigure}[t]{0.46\textwidth}
		\centering
		\includegraphics[width=\textwidth]{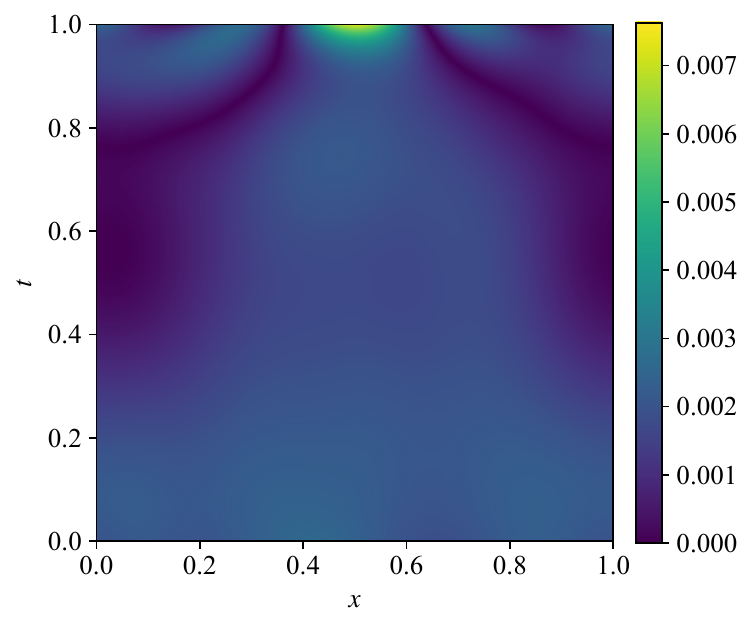}
		\caption{Absolute error}
		\label{fig:mfg_master_absolute_error}
	\end{subfigure}
	
	\caption{
		Verification of the characteristic relation along an MFG trajectory.
	}
	\label{fig:mfg_master_trajectory_comparison}
\end{figure}

\subsection{A systemic-risk problem with common noise}
\label{subsec:systemic_risk_common_noise}

To further test the proposed method in a practically motivated setting, we
consider a systemic-risk problem with common noise. In this model, the state
variable \(x\) represents the reserve level of a representative financial
institution, while the probability measure \(m\in\mathcal P(\Omega)\) describes
the distribution of reserve levels in the population. The state space is
\(\Omega=\mathbb R\), and the mean reserve level is denoted by
\[
\bar x_m:=\int_{\Omega} z\,m(dz).
\]
The running and terminal costs penalize deviations from the population mean,
and the associated feedback control models a stabilizing adjustment toward the
mean reserve level.

We consider the following master equation
\cite{bensoussan2015master}:
\begin{equation*}
	\left\{
	\begin{aligned}
		&-\partial_t U(t,x,m)
		-\frac{\sigma^2+\beta^2}{2}\Delta_x U(t,x,m)
		-\frac{\sigma^2+\beta^2}{2}
		\int_{\Omega}\operatorname{div}_y\!\left[D_mU(t,x,m,y)\right]m(dy)
		\\
		&\quad
		-\int_{\Omega}
		D_mU(t,x,m,y)
		\left[
		(\alpha+\lambda)(\bar x_m-y)-D_xU(t,y,m)
		\right]m(dy)
		-\beta^2\int_{\Omega}D_xD_mU(t,x,m,y)\,m(dy)
		\\
		&\quad
		-\frac{\beta^2}{2}
		\int_{\Omega}\int_{\Omega}
		\operatorname{div}_y\operatorname{div}_z
		\frac{\delta^2U}{\delta m^2}(t,x,m)(y,z)\,
		m(dy)m(dz)
		\\
		&=
		\frac{\mu-\lambda^2}{2}(\bar x_m-x)^2
		+(\alpha+\lambda)(\bar x_m-x)D_xU(t,x,m)
		-\frac12\left|D_xU(t,x,m)\right|^2,
		\\
		&\hspace{8.2cm}
		(t,x,m)\in [0,1)\times\Omega\times\mathcal P(\Omega),
		\\[0.5em]
		&U(1,x,m)=G(x,m):=\frac c2(\bar x_m-x)^2,
		\qquad (x,m)\in\Omega\times\mathcal P(\Omega).
	\end{aligned}
	\right.
\end{equation*}
Here \(\alpha>0\), \(\sigma>0\), \(\beta>0\), \(\lambda>0\),
\(\mu>\lambda^2\), and \(c>0\). Compared with the previous examples, the
common-noise equation contains two additional nonlocal contributions, involving
\(D_xD_mU\) and the second variation \(\delta^2U/\delta m^2\). Both
contributions are therefore included explicitly in the empirical residual used
for training. The equation and the explicit solution are defined on the full
space \(\Omega=\mathbb R\). In the numerical experiments, the residual
collocation points and error evaluation are restricted to the computational
window \(\Omega_{\rm num}:=[-2,2]\). Accordingly, the reported results are
evaluated on \([0,1]\times\Omega_{\rm num}\), and no artificial boundary
condition is imposed at \(x=\pm2\).

The two common-noise contributions require an extension of the
empirical residual in Section~\ref{subsec:residual_construction}. For a
PDE collocation point
\((t^{\mathrm{pde}}_{k,q},x^{\mathrm{pde}}_{k,q})\) associated with
\(m^{N,(k)}=N^{-1}\sum_{i=1}^N\delta_{\xi_i^{(k)}}\), the term \(\int_{\Omega}D_xD_mU(t,x,m,y)\,m(dy)\) is evaluated by
\begin{equation}
	\label{eq:systemic_extra_xm_empirical}
	J^{xm}_{\Theta,k,q}
	:=
	\frac1N\sum_{i=1}^N
	\partial_x D_mU_\Theta
	\bigl(t^{\mathrm{pde}}_{k,q},x^{\mathrm{pde}}_{k,q},m^{N,(k)},\xi_i^{(k)}\bigr).
\end{equation}
The second variation is obtained from the same MCOL architecture. More precisely, if
\(r=z_\eta(m)\), then a representative of the second variation is
\begin{equation*}
	\frac{\delta^2U_\Theta}{\delta m^2}(t,x,m)(y,z)
	=
	\phi_\eta(y)^\top
	D^2_{rr}\mathcal U_\Theta(t,x,r)\big|_{r=z_\eta(m)}
	\phi_\eta(z),
\end{equation*}
up to additive normalization terms that vanish after differentiating in
\((y,z)\). Hence the empirical approximation of the divergence
term is
\begin{equation}
	\label{eq:systemic_extra_mm_empirical}
	J^{mm}_{\Theta,k,q}
	:=
	\frac1{N^2}\sum_{i=1}^N\sum_{j=1}^N
	\partial_y\partial_z
	\frac{\delta^2U_\Theta}{\delta m^2}
	\bigl(t^{\mathrm{pde}}_{k,q},x^{\mathrm{pde}}_{k,q},m^{N,(k)}\bigr)
	(\xi_i^{(k)},\xi_j^{(k)}).
\end{equation}

Let \(P:[0,1]\to\mathbb R\) be the solution of the backward Riccati equation
\begin{equation*}
	\begin{cases}
		\displaystyle
		\frac{dP}{dt}
		-2(\alpha+\lambda)P
		-P^2
		+\mu-\lambda^2=0,
		\qquad 0\le t<1,\\[0.4em]
		P(1)=c .
	\end{cases}
\end{equation*}
Then the master equation admits the exact solution
\begin{equation*}
	U(t,x,m)
	=
	\frac12 P(t)\bigl(x-\bar x_m\bigr)^2
	+
	\frac{\sigma^2}{2}\int_t^1 P(s)\,ds .
\end{equation*}
The corresponding exact feedback control is
\begin{equation}
	\label{eq:systemic_risk_exact_feedback_control}
	v^*(t,x,m)
	=
	\bigl(\lambda+P(t)\bigr)(\bar x_m-x).
\end{equation}
The control pushes reserve levels below the mean upward and reserve
levels above the mean downward, which is consistent with the stabilizing
mean-reversion mechanism of the systemic-risk model.

Figure~\ref{fig:systemic_risk_U_vstar_slices} compares the exact and predicted
value functions and feedback controls for a representative empirical measure
\(m^N\) at \(t=0\), \(0.5\), and \(1\). The value function exhibits the expected
quadratic dependence on the deviation from the mean reserve level
\(\bar x_m\), while the feedback control is affine in \(x\) and vanishes at
\(x=\bar x_m\), as given by
\eqref{eq:systemic_risk_exact_feedback_control}. Its sign is consistent with
the mean-reversion mechanism of the systemic-risk model: institutions below the
population mean are driven upward, whereas those above the mean are adjusted
downward. The MCOL approximation agrees well with the exact profiles for both
\(U\) and \(v^*\) at all three time levels. In particular, the terminal slice
recovers the prescribed terminal cost, and the earlier slices show that the
learned operator captures the backward evolution of the value function. Since
the feedback control is obtained from \(D_xU_\Theta\), the agreement in \(v^*\)
also supports the accuracy of the learned state derivative. These results show
that the proposed method can handle the common-noise master equation associated
with the systemic-risk model on an unbounded state space.

\begin{figure}[htbp]
	\centering
	
	% ================= Left column: U =================
	\begin{subfigure}[t]{0.48\textwidth}
		\centering
		
		\includegraphics[width=\textwidth]{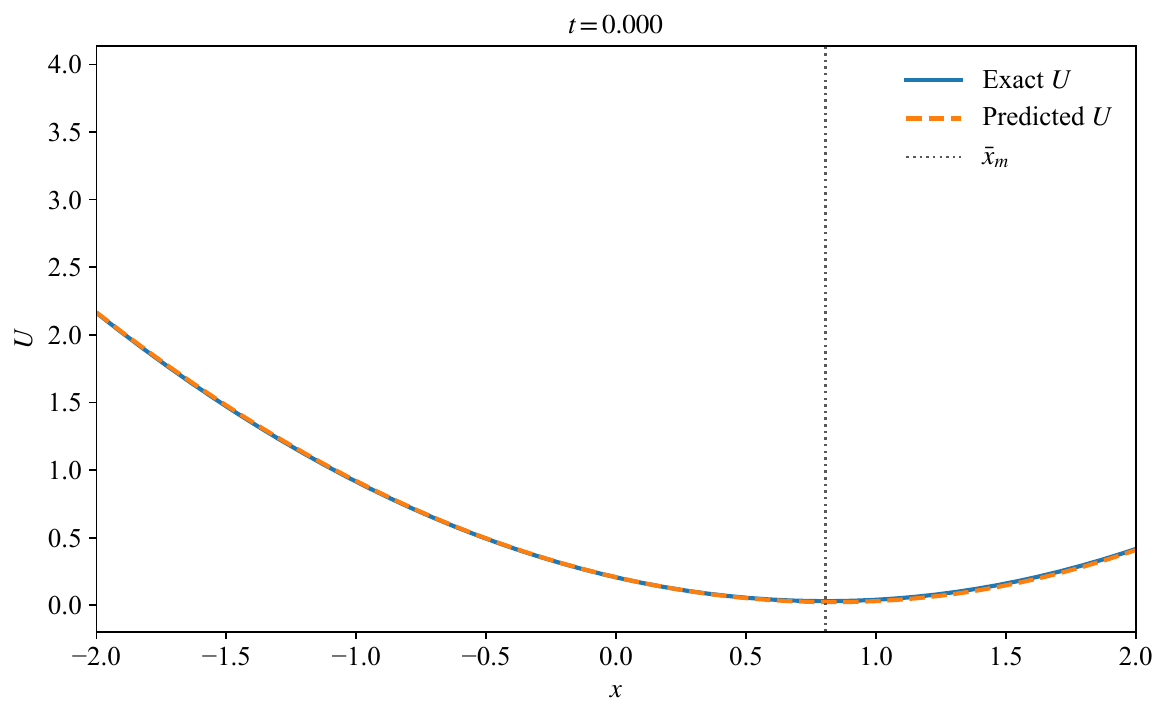}
		
		\vspace{0.35em}
		
		\includegraphics[width=\textwidth]{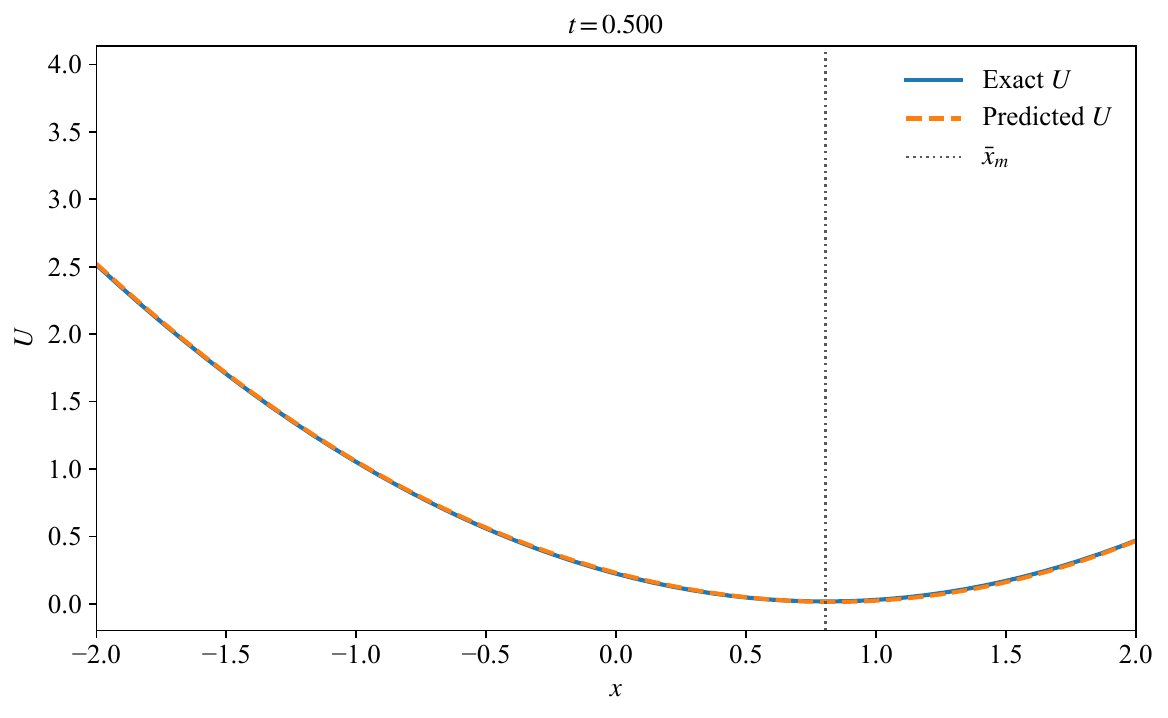}
		
		\vspace{0.35em}
		
		\includegraphics[width=\textwidth]{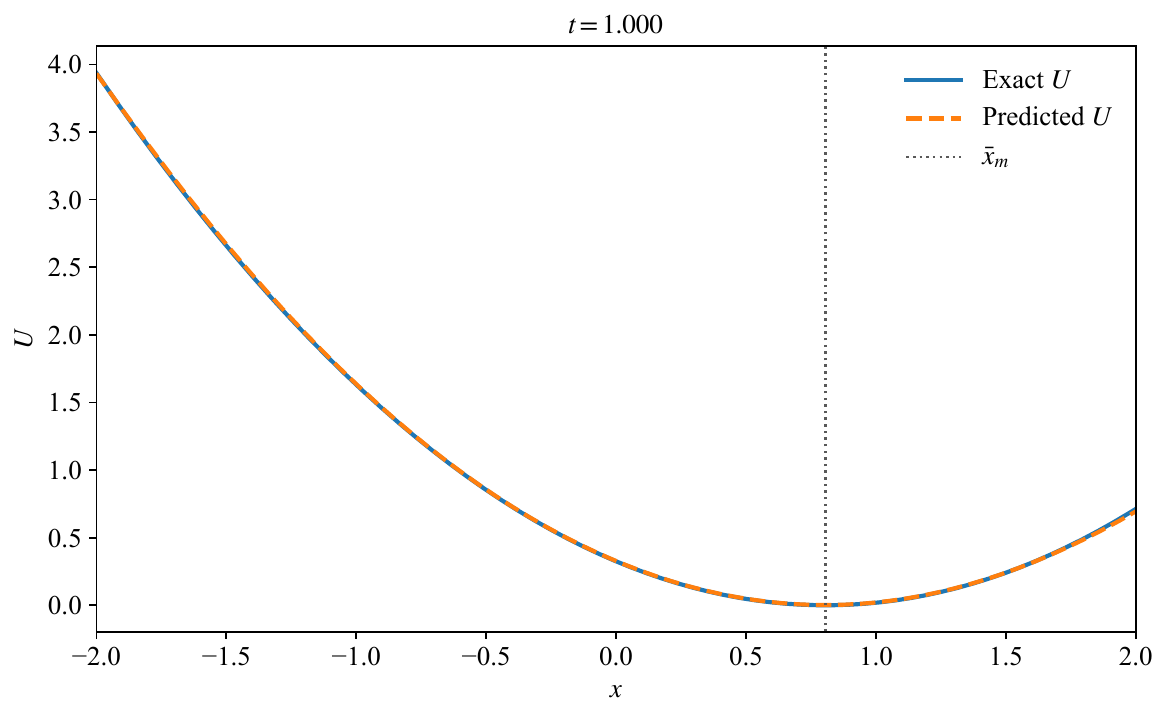}
		
		\caption{Value function \(U(t,x,m^N)\).}
		\label{fig:systemic_risk_U_slices}
	\end{subfigure}
	\hfill
	% ================= Right column: v* =================
	\begin{subfigure}[t]{0.48\textwidth}
		\centering
		
		\includegraphics[width=\textwidth]{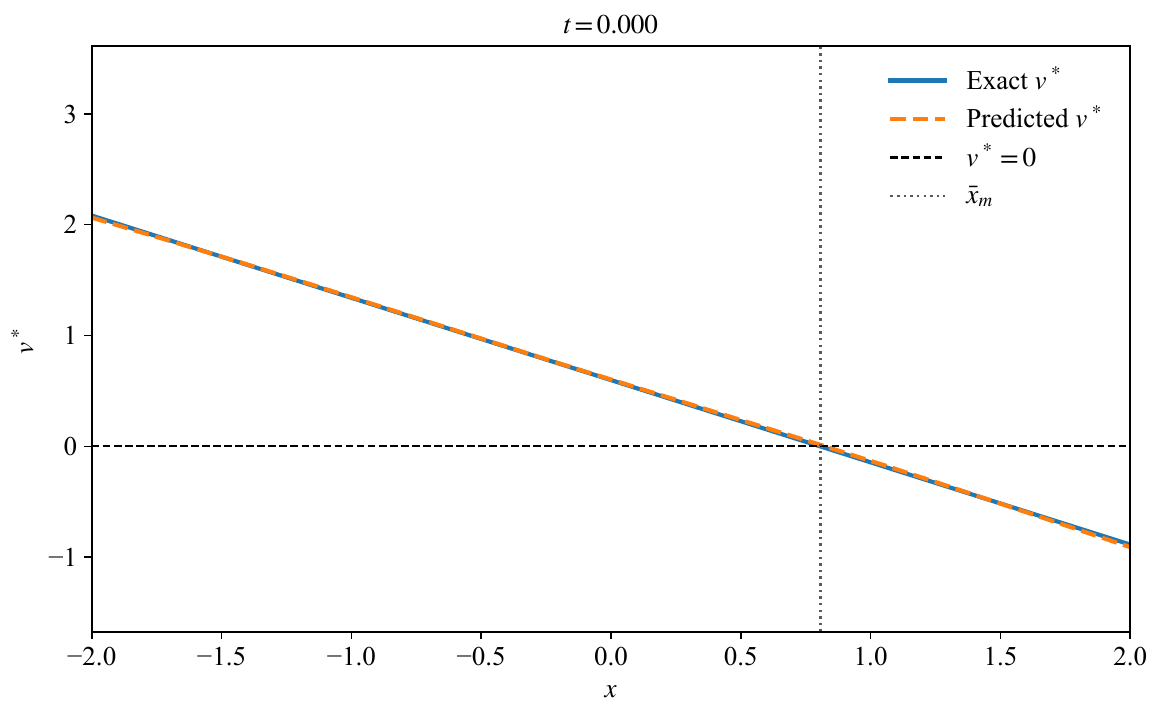}
		
		\vspace{0.35em}
		
		\includegraphics[width=\textwidth]{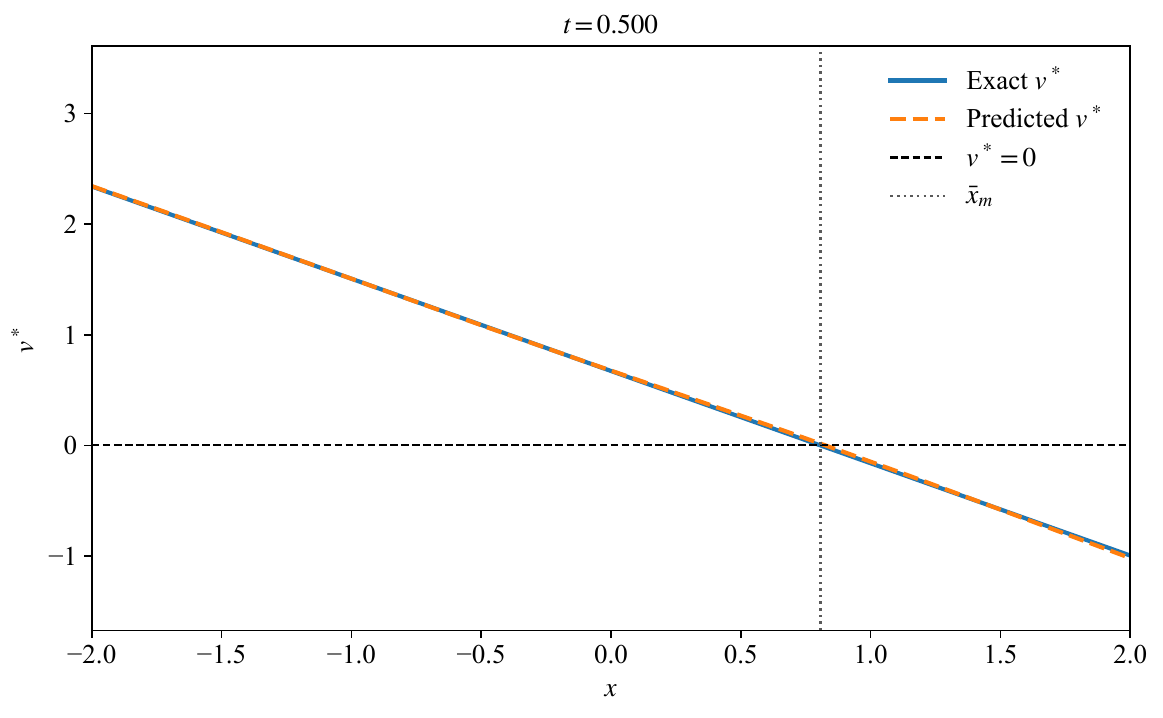}
		
		\vspace{0.35em}
		
		\includegraphics[width=\textwidth]{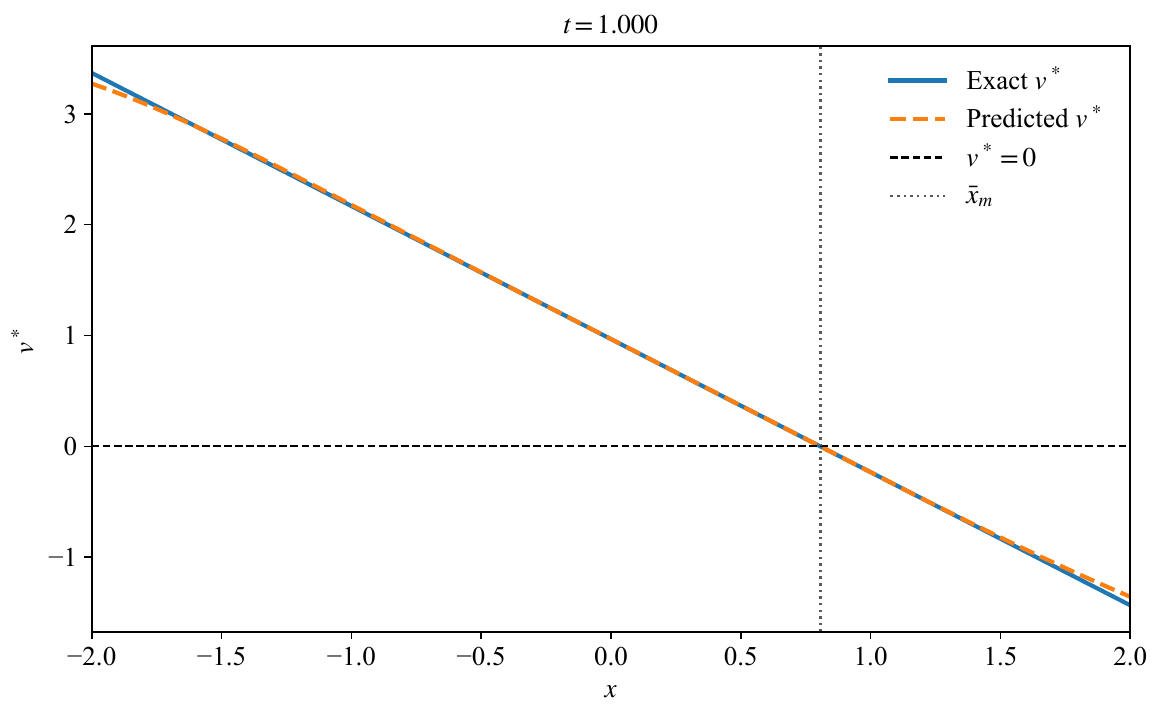}
		
		\caption{Optimal control \(v^*(t,x,m^N)\).}
		\label{fig:systemic_risk_vstar_slices}
	\end{subfigure}
	
	\caption{
		Time-slice comparisons of the value function and feedback control.
	}
	\label{fig:systemic_risk_U_vstar_slices}
\end{figure}

\section{Conclusion}
\label{sec:conclusion}

This paper developed a measure-consistent operator learning method for
infinite-dimensional master equations arising in MFG theory. The
population distribution is represented by empirical particles, and the same
empirical measure is used both as the network input and as the quadrature
measure for the nonlocal terms in the residual. A central feature of the method
is that the intrinsic derivative \(D_mU_\Theta\) is induced by the same
measure-dependent representation that defines the value approximation
\(U_\Theta\). Hence, the value function, its intrinsic measure derivative, and
the empirical residual are tied to a common representation of the measure
variable, rather than being approximated or assembled separately.

The numerical results show that this structure leads to accurate and stable
approximations across several representative master equations. The experiments
confirm that the proposed method can approximate both the value function and
the intrinsic derivative, and remain effective for empirical measures not seen during training. The reported error decomposition further separates the neural approximation error from the empirical discretization error, providing a clearer assessment of the
effect of particle resolution. These results indicate that enforcing consistency between the measure representation, the induced intrinsic derivative, and the residual assembly is beneficial for learning master equations. Future work will consider more general state spaces and adaptive sampling of empirical measures.

\appendix

\section{Baseline comparison}
\label{sec:baseline}

For comparison with MCOL, we introduce a parameterized PINN baseline with an auxiliary network for \(D_m\). This baseline is trained for the same master equation \eqref{eq:master_general_polished} under the same experimental setting, including the Hamiltonian, source term, terminal datum, boundary conditions, empirical training measures, collocation strategy, and optimizer schedule. The comparison therefore isolates the effect of the measure representation and the induced construction of the intrinsic derivative.

Let
\[
m^N=\frac1N\sum_{i=1}^N\delta_{\xi_i},
\qquad
\boldsymbol\xi=(\xi_1,\ldots,\xi_N)\in\Omega^N .
\]
The baseline represents the empirical measure by the particle vector \(\boldsymbol\xi\), and the value function is approximated by a fully connected network
\[
\widehat U_\Psi(t,x,\boldsymbol\xi)
\approx
U(t,x,m^N).
\]

To remain consistent with the implemented baseline and in the spirit of mixed residual formulations for high-order PDEs \cite{Lyu2022MIM}, we introduce an auxiliary neural network to approximate the intrinsic derivative:
\[
\widehat D_\Phi(t,x,y,\boldsymbol\xi)
\approx
D_mU(t,x,m^N,y),
\qquad
\widehat D_\Phi(t,x,y,\boldsymbol\xi)\in\mathbb R^d .
\]
In the baseline residual, \(U\) and \(D_mU\) are replaced by
\(\widehat U_\Psi\) and \(\widehat D_\Phi\), respectively, while the remaining discretization is the same as in the proposed MCOL method.

The main structural difference is that the baseline does not satisfy the measure-consistent relation between the value function and the intrinsic derivative. For an empirical measure \(m^N=N^{-1}\sum_{i=1}^N\delta_{\xi_i}\), this relation formally reads
\[
D_mU(t,x,m^N,\xi_i)
=
N\nabla_{\xi_i}U(t,x,m^N),
\qquad i=1,\ldots,N .
\]
In the proposed MCOL method, the corresponding identity is built into the measure encoder. In the baseline, however, \(\widehat U_\Psi\) and \(\widehat D_\Phi\) are two independent networks. Hence, the above relation is not built into the architecture and must be imposed only weakly through an additional consistency penalty.

For the empirical training measures
\[
m^{N,(k)}
=
\frac1N\sum_{i=1}^N\delta_{\xi_i^{(k)}},
\qquad
\boldsymbol\xi^{(k)}
=
(\xi_1^{(k)},\ldots,\xi_N^{(k)}),
\qquad
k=1,\ldots,M_{\mathrm{batch}},
\]
we define the chain-rule consistency loss by
\[
\widehat{\mathcal L}_{\mathrm{chain}}(\Psi,\Phi)
:=
\frac1{M_{\mathrm{batch}} Q_{\mathrm{pde}}N}
\sum_{k=1}^{M_{\mathrm{batch}}}
\sum_{q=1}^{Q_{\mathrm{pde}}}
\sum_{i=1}^N
\left|
\widehat D_\Phi
\bigl(
t^{\mathrm{pde}}_{k,q},
x^{\mathrm{pde}}_{k,q},
\xi_i^{(k)},
\boldsymbol\xi^{(k)}
\bigr)
-
N\nabla_{\xi_i}
\widehat U_\Psi
\bigl(
t^{\mathrm{pde}}_{k,q},
x^{\mathrm{pde}}_{k,q},
\boldsymbol\xi^{(k)}
\bigr)
\right|^2 .
\]
This term penalizes the mismatch between the explicitly learned intrinsic derivative and the particle gradient of the value network. It should therefore be interpreted as a weak consistency constraint, not as an architectural identity.

The total baseline loss is
\[
\widehat{\mathcal L}_{\mathrm{base}}(\Psi,\Phi)
=
\lambda_{\mathrm{pde}}
\widehat{\mathcal L}_{\mathrm{pde}}(\Psi,\Phi)
+
\lambda_T
\widehat{\mathcal L}_{T}(\Psi)
+
\lambda_{\partial x}
\widehat{\mathcal L}_{\partial x}(\Psi)
+
\lambda_{\partial m}
\widehat{\mathcal L}_{\partial m}(\Phi)
+
\lambda_{\mathrm{chain}}
\widehat{\mathcal L}_{\mathrm{chain}}(\Psi,\Phi).
\]
Here, \(\widehat{\mathcal L}_{\mathrm{pde}}\), \(\widehat{\mathcal L}_{T}\), \(\widehat{\mathcal L}_{\partial x}\), and \(\widehat{\mathcal L}_{\partial m}\) are defined in the same way as in the proposed MCOL method after replacing \(U_\Theta\) by \(\widehat U_\Psi\) and \(D_mU_\Theta\) by \(\widehat D_\Phi\).

\section*{Acknowledgments}
The research was supported by the Hong Kong RGC General Research Funds (projects 11311122, 12301420, and 11300821), the National Natural Science Foundation of China (grant 12571424), and the Shenzhen Science and Technology Program (grants RCJC20210609103755110 and JCYJ20240813104914020).

\bibliographystyle{elsarticle-num}
\bibliography{ref1}

\end{document}